\documentclass[12pt]{amsart}
\usepackage[margin=1in]{geometry}
\usepackage{amsmath}
\usepackage{amsthm}
\usepackage{amsbsy}
\usepackage{amssymb}
\usepackage{hyperref}
\usepackage{amsfonts}
\usepackage{color}
\usepackage{constants}

\usepackage{enumitem}
\numberwithin{equation}{section}
\newtheorem{theorem}{Theorem}[section]
\newtheorem{lemma}[theorem]{Lemma}
\newtheorem{proposition}[theorem]{Proposition}
\newtheorem{corollary}[theorem]{Corollary}

\newtheorem*{hypothesis*}{The generalized Ramanujan conjecture}
\newtheorem*{definition*}{Definition}

\theoremstyle{remark}
\newtheorem*{remark}{Remark}

\theoremstyle{remark}
\newtheorem*{remark1}{Remarks}

\theoremstyle{remark}

\newcommand{\Z}{\mathbb{Z}}
\newcommand{\R}{\mathbb{R}}
\newcommand{\Q}{\mathbb{Q}}

\newcommand{\p}{\mathfrak{p}}

\renewcommand{\v}{\mathfrak{v}}

\newcommand{\re}{\textup{Re}}
\newcommand{\im}{\textup{Im}}

\renewcommand{\a}{\mathfrak{a}}

\newcommand{\q}{\mathfrak{q}}

\renewcommand{\pmod}[1]{\left(\mathrm{mod}\,\,#1\right)}

\newconstantfamily{abcon}{symbol=c}

\title[Effective log-free zero density estimates]{Effective log-free zero density estimates for automorphic $L$-functions and the Sato-Tate conjecture}
\author{Robert J. Lemke Oliver}
\address{Department of Mathematics, Tufts University, Medford, MA 02155}
\email{robert.lemke\_oliver@tufts.edu}
\author{Jesse Thorner}
\address{Department of Mathematics, Stanford University, Stanford, CA 94305}
\email{jthorner@stanford.edu}
\date{\today}
\thanks{The first author was supported by an NSF Mathematical Sciences Postdoctoral Research Fellowship and is presently supported by NSF grant DMS-1601398.  The second author is supported by an NSF Mathematical Sciences Postdoctoral Research Fellowship.}
\begin{document}

\begin{abstract}
Let $K/\Q$ be a number field.  Let $\pi$ and $\pi^\prime$ be cuspidal automorphic representations of $\textup{GL}_d(\mathbb{A}_K)$ and $\textup{GL}_{d^\prime}(\mathbb{A}_K)$, and suppose that either both $d$ and $d'$ are at most 2 or at least one of $\pi$ and $\pi^\prime$ is self-dual.  When $d=d^\prime=2$, we prove an unconditional and effective log-free zero density estimate for the Rankin-Selberg $L$-function $L(s,\pi\otimes\pi^\prime,K)$.  For other choices of $d$ and $d^\prime$, we obtain similar results assuming that either $\pi$ or $\pi^\prime$ satisfies the generalized Ramanujan conjecture.  With these density estimates, we make effective the Hoheisel phenomenon of Moreno regarding primes in short intervals and extend it to the context of the Sato-Tate conjecture; additionally, we bound the least prime in the Sato-Tate conjecture in analogy with Linnik's theorem on the least prime in an arithmetic progression.  When $K=\Q$, we also prove an effective log-free density estimate for $L(s,\pi\otimes\pi^\prime,\Q)$ averaged over twists by Dirichlet characters.  With this second density estimate, we prove an averaged form of the prime number theorem in short intervals for $L(s,\pi\otimes\tilde{\pi},\Q)$ when $\pi$ is a cuspidal automorphic representation of $\textup{GL}_2(\mathbb{A}_{\Q})$.
\end{abstract}

\maketitle

\section{Introduction and statement of results}
\label{sec:Section1}

The classical prime number theorem asserts that
\[
\sum_{n\leq x} \Lambda(n) \sim x,
\]
where $\Lambda(n)$ is the von Mangoldt function.  Depending on the quality of the error term, it is possible to deduce from this a prime number theorem for short intervals, in the form
\begin{equation}
\label{eqn:pnt_short}
\sum_{x<n\leq x+h} \Lambda(n) \sim h,
\end{equation}
provided that $h$ is not too small; with the presently best known error terms, we may take $h$ a bit smaller than $x$ divided by any power of $\log x$, but not as small as $x^{1-\delta}$ for any $\delta>0$.  Improving the error bound in the prime number theorem to allow for $h$ to be of size $x^{1-\delta}$ is a monumentally hard task, known as the quasi-Riemann hypothesis, and amounts to showing that there are no zeros of the Riemann zeta function $\zeta(s)$ in the region $\Re(s) >1-\delta$. Nevertheless, in 1930, Hoheisel \cite{Hoheisel} made the remarkable observation that, with Littlewood's improved zero-free region for $\zeta(s)$, if there are simply \emph{not too many} zeros in this region, then one can deduce \eqref{eqn:pnt_short} with $h = x^{1-\delta}$.  In particular, it turns out that
\begin{equation}
\label{eqn:hoheisel_zerodensity}
N(\sigma,T) := \# \{ \rho=\beta+i\gamma\colon\zeta(\rho)=0, \beta \geq \sigma, |\gamma|\leq T \} \ll T^{c(1-\sigma)}(\log T)^{c'},
\end{equation}
where $c>2$ and $c'>0$ are absolute constants; this is a so-called {\bf zero density estimate}. (In this section, $c$ and $c'$ will always denote positive absolute constants, though they may represent different values in each occurrence.)  Recall that there are about $\frac{T}{\pi} \log \frac{T}{2\pi e}$ zeros of $\zeta(s)$ with $|\gamma|\leq T$, so that a vanishingly small proportion of zeros have real part close to $1$.  An explicit version of \eqref{eqn:hoheisel_zerodensity} enabled Hoheisel to prove the prime number theorem in short intervals \eqref{eqn:pnt_short} for $h=x^{1-\delta}$ in the range $0\leq \delta \leq 1/33000$; it is now known that we may take $0\leq \delta\leq\frac{5}{12}$, due to Huxley \cite{Huxley} and Heath-Brown \cite{Heath-Brown1988}.

Another classical problem in analytic number theory is to determine the least prime in an arithmetic progression $a\pmod{q}$ with $(a,q)=1$.  Linnik \cite{Linnik} was able to show that the least such prime is no bigger than $q^A$, where $A$ is an absolute constant; the best known value of $A$ is 5, due to Xylouris \cite{Xylouris} in his Ph.D. thesis.  Modern treatments of Linnik's theorem typically use a simplification due to Fogels \cite{Fogels}, which involves proving a more general version of \eqref{eqn:hoheisel_zerodensity} for Dirichlet $L$-functions $L(s,\chi)$.  Specifically, if we define
\[
N_\chi(\sigma,T) := \#\{ \rho=\beta+i\gamma\colon L(\rho,\chi)=0, \beta\geq \sigma, \text{ and } |\gamma|\leq T\},
\]
then Fogels showed that
\begin{equation}
\label{eqn:fogels}
\sum_{\chi\pmod{q}} N_\chi(\sigma,T) \ll T^{c(1-\sigma)}
\end{equation}
when $T\geq q$.  Due to the absence of a $\log T$ term as compared to \eqref{eqn:hoheisel_zerodensity}, it is standard to call such a result a {\bf log-free zero density estimate}.  In this paper, we are interested in analogous log-free zero density estimates for automorphic $L$-functions and their arithmetic applications, specifically to analogues of Hoheisel's and Linnik's theorems.  

We consider the following general setup.  Let $K/\Q$ be a number field with ring of adeles $\mathbb{A}_K$, and let $\mathcal{A}_d(K)$ denote the set of all cuspidal automorphic representations of $\mathrm{GL}_d(\mathbb{A}_K)$ with unitary central character. If $\pi\in\mathcal{A}_d(K)$, then there is an $L$-function $L(s,\pi,K)$ attached to $\pi$ whose  Dirichlet series and Euler product are given by
\begin{equation*}
\label{auto-L-func}
L(s,\pi,K)=\sum_{\a}\frac{\lambda_{\pi}(\a)}{\mathrm{N}\a^s}=\prod_{\p}\prod_{j=1}^d(1-\alpha_{\pi}(j,\p)\mathrm{N}\p^{-s})^{-1},
\end{equation*}
where the sum runs over the non-zero integral ideals of $K$, the product runs over the prime ideals of $K$, and $\mathrm{N}\a=\mathrm{N}_{K/\mathbb{Q}}\a$ denotes the norm of the ideal $\a$.

Let $\pi\in\mathcal{A}_d(K)$ and $\pi^\prime\in\mathcal{A}_{d^\prime}(K)$.  The Rankin-Selberg convolution
\begin{equation*}
L(s,\pi\otimes\pi',K)=\sum_{\a}\frac{\lambda_{\pi\otimes\pi'}(\a)}{\mathrm{N}\a^s}=\prod_{\p}\prod_{j_1=1}^{d} \prod_{j_2=1}^{d'}(1-\alpha_{\pi}(j_1,\p)\alpha_{\pi'}(j_2,\p)\mathrm{N}\p^{-s})^{-1}
\end{equation*}
is itself an $L$-function with an analytic continuation and a functional equation.  Define $\Lambda_{\pi\otimes\pi'}(\mathfrak{a})$ by the Dirichlet series identity
\begin{equation*}
-\frac{L^\prime}{L}(s,{\pi\otimes\pi'},K) = \sum_{\mathfrak{a}} \frac{\Lambda_{\pi\otimes\pi'}(\mathfrak{a})}{\mathrm{N} \mathfrak{a}^{s}}.
\end{equation*}
If $\tilde{\pi}$ is the representation which is contragredient to $\pi$, then it follows from standard Rankin-Selberg theory and the Wiener-Ikehara Tauberian theorem that we have a prime number theorem for $L(s,\pi\otimes\tilde{\pi},K)$ in the form
\[
\sum_{\mathrm{N}\a \leq x} \Lambda_{\pi\otimes\tilde{\pi}}(\a) \sim x.
\]

It is reasonable to expect (for example, it follows from the generalized Riemann hypothesis) that there is some small $\delta>0$ such that for $x$ sufficiently large and any $h\geq x^{1-\delta}$, we have
\begin{equation}
\label{eqn:auto_hoheisel}
\sum_{x<\mathrm{N}\a\leq x+h}\Lambda_{\pi\otimes\tilde{\pi}}(\a) \sim h.
\end{equation}
Unfortunately, a uniform analogue of Littlewood's improved zero-free region does not yet exist for all automorphic $L$-functions, so it seems that \eqref{eqn:auto_hoheisel} is currently inaccessible except in special situations.  However, it follows from the work of Moreno \cite{Moreno} that if $L(s,\pi\otimes\tilde{\pi},K)$ has a ``standard'' zero-free region (one of a quality similar to Hadamard's and de la Vall\'ee Poussin's for $\zeta(s)$, see Lemma \ref{1.9}), and if there is a log-free zero density estimate of the form
\begin{equation*}
\label{eqn:Npichi}
N_{\pi\otimes\pi^\prime}(\sigma,T):=\#\{\rho=\beta+i\gamma\colon L(\rho,\pi\otimes\pi^\prime,K)=0,\beta\geq \sigma,|\gamma|\leq T\}\ll T^{c_{\pi,\pi'}(1-\sigma)}
\end{equation*}
for $L(s,\pi\otimes\tilde{\pi},K)$, then for any $0<\delta<1/c_{\pi,\tilde{\pi}}$ and any $h\geq x^{1-\delta}$, one has
\begin{equation}
\label{eqn:twisted_PNT_pi}
\sum_{x<\mathrm{N}\a\leq x+h}\Lambda_{\pi\otimes\tilde{\pi}}(\a)\gg h,
\end{equation}
which Moreno called the {\bf Hoheisel phenomenon}.    However, at the time of Moreno's work, such log-free zero density estimates only existed in special cases.  Moreover, in general, it is only known that $L(s,\pi\otimes\tilde{\pi},K)$ has a standard zero-free region if $\pi$ is self-dual.  

Recall that $\pi\in\mathcal{A}_d(K)$ and $\pi^\prime\in\mathcal{A}_{d^\prime}(K)$.  Suppose that $K=\mathbb{Q}$ and that either $d$ and $d'$ are both at most 2 or that one of $\pi$ and $\pi'$ is self-dual.  Building on the work of Fogels, Akbary and Trudgian \cite{AT} proved in this case that if one has a certain amount of control over the Dirichlet coefficients of $L(s,\pi\otimes\tilde{\pi},\Q)$ and $L(s,\pi^\prime\otimes\tilde{\pi}^\prime,\Q)$ in short intervals (see Hypothesis 1.1 of \cite{AT}) and $T$ is sufficiently large in terms of $\pi$ and $\pi'$, then
\[
N_{\pi\otimes\pi'}(\sigma,T)\leq T^{c_{d,d'}(1-\sigma)},
\]
where $c_{d,d'}>2$ is a constant depending on $d$ and $d'$.  This allowed them to prove a variant of the Hoheisel phenomenon for $L(s,\pi\otimes\tilde{\pi},\Q)$ when $\pi$ is self-dual.  Unfortunately, the constant $c_{d,d'}$ was not made effective, whence also the length of the interval in their variant of the Hoheisel phenomenon.  This makes their result difficult to use in situations where uniformity in parameters over several $L$-functions is required, especially when the $L$-functions in question vary in degree.  Furthermore, the range of $T$ for which their bound holds is also not made effective.  This is necessary to obtain analogues of Linnik's theorem.

Effective log-free zero density estimates have been proven for certain natural families of $L$-functions.  Weiss \cite{weiss} proved an effective analogue of \eqref{eqn:fogels} for the Hecke $L$-functions of ray class characters, which enabled him to access prime ideals of $K$ satisfying splitting conditions in a finite Galois extension $M/K$.  Additionally, Kowalski and Michel \cite{KM} obtained a log-free zero density estimate for $L$-functions associated to any family of automorphic representations of $\mathrm{GL}_d(\mathbb{A}_{\Q})$ satisfying certain conditions, including the generalized Ramanujan conjecture (see Hypothesis \ref{GRC}).  Their result works best when $T$ is essentially constant, which is useful for variants of Linnik's theorem but not for the Hoheisel phenomenon.

We prove several log-free zero density estimates for Rankin-Selberg $L$-functions $L(s,\pi\otimes\pi',K)$ with effective dependence on $\pi$, $\pi'$, and $K$.  This dependence is most naturally stated in terms of the analytic conductors $\q(\pi)$ and $\q(\pi^\prime)$ of $\pi$ and $\pi'$, respectively (see \eqref{eqn:analytic_conductor}).

\begin{theorem}
\label{main-theorem}
Let $K$ be a number field with absolute discriminant $D_K$.  Let $\pi\in\mathcal{A}_d(K)$ and $\pi'\in\mathcal{A}_{d^\prime}(K)$.  Suppose either that at least one of $d$ and $d'$ equals one or that at least one of $\pi$ and $\pi'$ is self-dual, and suppose that $\pi'$ satisfies the generalized Ramanujan conjecture (GRC).  Let $\mathcal{Q}=\mathcal{Q}(\pi,\pi^\prime,K)$ be defined by
\[
\mathcal{Q}=\begin{cases}
\max\{\q(\pi),D_K[K:\Q]^{[K:\Q]}\}&\mbox{if $\pi^\prime$ is the trivial representation of $\mathrm{GL}_1(\mathbb{A}_K)$},\\
\max\{\q(\pi)\q(\pi^\prime),D_K[K:\Q]^{[K:\Q]}\}&\mbox{otherwise},
\end{cases}
\]
let
\[
\mathcal{D}=\mathcal{D}(\pi,\pi')=\begin{cases}
d^2&\mbox{if $d=d'$ and both $\pi$ and $\pi'$ are self-dual},\\
d^4&\mbox{if $\pi^\prime$ is the trivial representation of $\mathrm{GL}_1(\mathbb{A}_K)$},\\
(d+d')^4&\mbox{otherwise,}
\end{cases}
\]
and let $T\geq1$.  There exists an absolute constant $\Cl[abcon]{c1}>0$ such that if $\frac{1}{2}\leq\sigma\leq 1$, then\footnote{Unless mentioned otherwise, the implied constant in an asymptotic inequality is absolute and computable.}
\[
N_{\pi\otimes\pi'}(\sigma,T)\ll (d')^2 (\mathcal{Q}T^{[K:\Q]})^{\Cr{c1} \mathcal{D}(1-\sigma)}.
\]
\end{theorem}

The next result follows unconditionally from Theorem \ref{main-theorem} by letting $\pi^\prime$ be the trivial representation of $\mathrm{GL}_1(\mathbb{A}_K)$.

\begin{corollary}
\label{cor:lfzde}
Let $K$ be a number field, and let $\pi\in\mathcal{A}_d(K)$.  If $T\geq1$ and $\frac{1}{2}\leq\sigma\leq 1$, then
\[
N_{\pi}(\sigma,T)\ll (\mathcal{Q}T^{[K:\Q]})^{\Cr{c1}d^4(1-\sigma)}.
\]
\end{corollary}

\begin{remark1}
1.  We impose the self-duality condition in Theorem \ref{main-theorem} in order to ensure that $L(s,\pi\otimes\pi',K)$ has a standard zero-free region; see Lemma \ref{1.9}.

2.  Corollary \ref{cor:lfzde} is the first unconditional log-free zero density estimate for all automorphic $L$-functions $L(s,\pi,K)$.  Recall that Akbary and Trudgian's result requires $K=\mathbb{Q}$ and is conditional on a hypothesis on the Dirichlet coefficients of $L(s,\pi,\mathbb{Q})$ in short intervals.  In fact, using Theorem \ref{main-theorem} and Corollary \ref{cor:lfzde} (whose proofs do not require this hypothesis), we can show that the hypothesis is satisfied in many cases.  See the remarks after Theorem \ref{effective-hoheisel} and equations \eqref{eqn:shakedown_street}-\eqref{eqn:bound_sum_zeros_LFZDE}.
\end{remark1}

In addition to density estimates of the form \eqref{eqn:fogels}, Jutila \cite{Jutila} proved a ``hybrid'' density estimate of the form
\begin{equation}
\label{eqn:large_sieve_density}
\sum_{q\leq Q}~\sideset{}{^\star}\sum_{\chi\bmod q}N_{\chi}(\sigma,T)\ll (Q^2 T)^{c(1-\sigma)}(\log QT)^{c'},
\end{equation}
where the $^\star$ on the summation indicates it is to be taken over primitive characters.  Montgomery \cite{MR0249375} improved upon Jutila's work to show that one may take $c=\frac{5}{2}$.  This simultaneously generalizes \eqref{eqn:fogels} and Bombieri's large sieve density estimate \cite{Bombieri}.  As a consequence of \eqref{eqn:large_sieve_density}, one sees that the average value of $N_{\chi}(\sigma,T)$ is noticeably smaller than what is given by \eqref{eqn:fogels}.  Furthermore, \eqref{eqn:large_sieve_density} can be used to prove versions of the Bombieri-Vinogradov theorem in both long and short intervals.

Gallagher \cite{Gallagher2} proved that
\begin{equation}
\label{eqn:Gallagher}
\sum_{q\leq T}~\sideset{}{^\star}\sum_{\chi\bmod q}N_{\chi}(\sigma,T)\ll T^{c(1-\sigma)},\qquad T\geq 1,
\end{equation}
providing a mutual refinement of \eqref{eqn:fogels} and \eqref{eqn:large_sieve_density}.  Gallagher's refinement can be also used to prove Linnik's bound on the least prime in an arithmetic progression.  Our second result generalizes \eqref{eqn:Gallagher} to consider twists of Rankin-Selberg $L$-functions associated to automorphic representations over $\Q$.

\begin{theorem}
\label{main-theorem-2}
Under the notation and hypotheses of Theorem \ref{main-theorem} with $K=\Q$, there exists an absolute constant $\Cl[abcon]{c1'}>0$ such that
\[
\sum_{q\leq T}~\sideset{}{^\star}\sum_{\chi\bmod q}N_{(\pi\otimes\chi)\otimes\pi'}(\sigma,T)\ll (d')^2 (\mathcal{Q}T)^{\Cr{c1'}\mathcal{D}(1-\sigma)}.
\]
\end{theorem}

As with Theorem \ref{main-theorem}, we immediately obtain the following unconditional corollary by letting $\pi^\prime$ be the trivial representation of $\mathrm{GL}_1(\mathbb{A}_{\Q})$.

\begin{corollary}
\label{cor:lfzde2}
Under the notation and hypotheses of Corollary \ref{cor:lfzde} with $K=\Q$, we have that
\[
\sum_{q\leq T}~\sideset{}{^\star}\sum_{\chi\bmod q}N_{\pi\otimes\chi}(\sigma,T)\ll (\mathcal{Q}T)^{\Cr{c1'}d^4(1-\sigma)}.
\]
\end{corollary}

We can sometimes circumvent the need for GRC or self-duality in Theorem \ref{main-theorem} by using Corollary \ref{cor:lfzde} along with certain advances toward the Langlands program.  For example, our next result shows that there is always a log-free zero density estimate for $L(s,\pi\otimes\pi^\prime,K)$ whenever $\pi,\pi^\prime \in \mathcal{A}_2(K)$, with no additional hypotheses needed.

\begin{theorem}
\label{thm:no_GRC}
	Let $K$ be a number field.  Let $\pi\in\mathcal{A}_d(K)$ and $\pi'\in\mathcal{A}_{d'}(K)$ with $d,d'\in\{1,2\}$.  Define $\mathcal{Q}$ as in Theorem \ref{main-theorem}, and let $T\geq1$.  If $\frac{1}{2}\leq\sigma\leq1$, then
	\[
	N_{\pi\otimes\pi'}(\sigma,T)\ll (\mathcal{Q}T^{[K:\Q]})^{64\Cr{c1}(1-\sigma)}.
	\]
	If $K=\Q$, then
	\[
	\sum_{q\leq T}~\sideset{}{^\star}\sum_{\chi\bmod q}N_{(\pi\otimes\chi)\otimes\pi'}(\sigma,T)\ll (\mathcal{Q}T^{[K:\Q]})^{64\Cr{c1'}(1-\sigma)}.
	\]
\end{theorem}
\begin{remark}
	In particular, Theorem \ref{thm:no_GRC} applies to $L(s,\pi\otimes\pi',K)$, where $\pi,\pi'\in\mathcal{A}_2(K)$ each correspond with Hecke-Maass forms for which GRC is not known.  The special case where $K=\Q$, $\pi$ corresponds to a Hecke-Maass form, and $\pi'\cong\tilde{\pi}$ was proved by Motohashi \cite{Motohashi} using methods different from our own.
\end{remark}

We now turn to the applications of Theorems \ref{main-theorem} and \ref{main-theorem-2} and their corollaries.  We begin by considering a version of \eqref{eqn:twisted_PNT_pi} with effective bounds on the size of the intervals for $L$-functions satisfying the generalized Ramanujan conjecture.  

\begin{theorem}
\label{effective-hoheisel}
Assume the above notation.  Let $\pi\in\mathcal{A}_d(K)$ be a self-dual representation which satisfies GRC.  There exists a positive absolute constant $\Cl[abcon]{hoheiselc}>0$ such that if
\[
\delta\leq\frac{\Cr{hoheiselc}}{d^{4}[K:\Q]\log(3d)},
\]
$x$ is sufficiently large, and $h\geq x^{1-\delta}$, then
\[
\sum_{x<\mathrm{N}\a\leq x+h}\Lambda_{\pi\otimes\tilde{\pi}}(\a)\asymp h,
\]
where the implied constant depends on $\pi$ and $K$.  If $\pi\in\mathcal{A}_2(K)$, then the result is unconditional and depends on neither GRC nor self-duality.
\end{theorem}
\begin{remark}
When $L(s,\pi\otimes\tilde{\pi},K)$ factors as a product of $L$-functions of cuspidal automorphic representations, then our proof of Theorem \ref{effective-hoheisel} {\it confirms} Hypothesis 1.1 of \cite{AT}.  This is particularly interesting when $\pi$ is associated to a Hecke-Maass form over $K$, where GRC is not known.  In this case, however, when $K=\mathbb{Q}$, Motohashi \cite{Motohashi} proved a version of Theorem \ref{effective-hoheisel} using his aforementioned log-free zero density estimate.
\end{remark}

It is of course somewhat unsatisfying that we are not able to obtain an asymptotic formula in Theorem \ref{effective-hoheisel} to provide a true short interval analogue of \eqref{eqn:twisted_PNT_pi}.  As remarked earlier, this is due to the lack of a strong zero-free region for general automorphic $L$-functions and seems unavoidable at present.  Good zero-free regions of a quality better than Littlewood's exist for Dedekind zeta functions (for example, due to Mitsui \cite{Mitsui}), which enabled Balog and Ono \cite{BO} to prove a prime number theorem for prime ideals in Chebotarev sets lying in short intervals.

Even though versions of Theorem \ref{effective-hoheisel} with asymptotic equality are only known in special cases, we can use Theorem \ref{main-theorem-2} to show that the predicted asymptotic holds on average.  We prove the following generalization of \cite[Theorem 7]{Gallagher2}; to obtain unconditional results, we restrict ourselves to consider automorphic representations of $\mathrm{GL}_2(\mathbb{A}_{\Q})$.

\begin{theorem}
\label{thm:auto_BV}
Let $\pi\in\mathcal{A}_2(\Q)$.  There exist absolute constants $\Cl[abcon]{BV1}>0$ and $\Cl[abcon]{BV2}>0$ such that if $\exp(\sqrt{\log x})\leq Q\leq x^{\Cr{BV1}}$ and $x/Q\leq h\leq x$, then
\begin{align*}
\sum_{q\leq Q}~\sideset{}{^\star}\sum_{\chi\bmod q}\left|\sum_{x<n\leq x+h}\Lambda_{\pi\otimes\tilde{\pi}}(n)\chi(n)-\delta(\chi)h+\delta_{q,*}(\chi)h \xi^{\beta_{\mathrm{exc}}-1}\right|\ll h\exp\Big(- \frac{\Cr{BV2}\log x}{\log(Q\q(\pi))}\Big)
\end{align*}
for some $\xi\in[x,x+h]$.  Here, $\delta(\chi)=1$ if $\chi$ is the trivial character and is zero otherwise, and $\beta_{\mathrm{exc}}$ denotes the Siegel zero associated to an exceptional real Dirichlet character $\chi^* \pmod{q}$ if it exists.   We set $\delta_{q,*}(\chi)=1$ if $\chi=\chi^*$ and zero otherwise.  The implied constant depends on at most $\q(\pi)$.  
\end{theorem}

Unlike the previous log-free zero density estimates for general automorphic $L$-functions discussed earlier, Theorem \ref{main-theorem} allows us to handle questions where maintaining uniformity in parameters is crucial. One famous example of such an application is the Sato-Tate conjecture, which concerns the distribution of the quantities $\lambda_\pi(\mathfrak{p})$ attached to representations $\pi\in\mathcal{A}_2(K)$, where $K$ is a totally real field;  for generalizations to higher degree representations, see, for example, Serre \cite{SerreLectures}.  Suppose that $\pi$ has trivial central character, does not have complex multiplication, and is \emph{genuine} (see Section \ref{subsec:sato-tate} for a definition).  Suppose further that $\pi$ satisfies GRC.  Then $|\lambda_\pi(\mathfrak{p})| \leq 2$ at all unramified $\p$.  We may thus write $\lambda_{\pi}(\mathfrak{p})=2\cos\theta_{\p}$ for some angle $\theta_{\p}\in[0,\pi]$.  The Sato-Tate conjecture predicts that if $I=[a,b]\subset[-1,1]$ is a fixed subinterval, then
\begin{equation*}
\label{eqn:satotate}
\lim_{x\to\infty}\frac{1}{\pi_K(x)} \#\{\mathrm{N}\mathfrak{p} \leq x\colon \cos \theta_{\p} \in I\} = \frac{2}{\pi} \int_I \sqrt{1-t^2}\,dt =: \mu_{\mathrm{ST}}(I),
\end{equation*}
where $\pi_K(x):=\#\{\mathfrak{p}\colon\mathrm{N}\mathfrak{p}\leq x\}$.  The Sato-Tate conjecture is now a theorem for large classes of $\pi$.  For newforms over $\Q$ and elliptic curves over totally real fields, this was proved by Barnet-Lamb, Geraghty, Harris, and Taylor \cite{Sato-Tate}, and for Hilbert modular forms, this was done by Barnet-Lamb, Gee, and Geraghty \cite{BGG}.  The proofs rely upon showing that the symmetric power $L$-functions $L(s,\mathrm{Sym}^n \pi,K)$ are all potentially automorphic, that is, there exists a finite, totally real Galois extension $L/K$ such that $\mathrm{Sym}^n \pi$ is automorphic over $L$.  It is expected that $L(s,\mathrm{Sym}^n \pi,K)\in\mathcal{A}_{n+1}(K)$ for each $n\geq1$, but as of right now, this is known in general only for $n\leq 4$ (see \cite{GJ,Kim,KS1,KS2}).  By recent work of Clozel and Thorne \cite{ClozelThorne}, if $\pi$ is associated to a classical modular form, and $K\cap\Q(e^{2\pi i/35})=\Q$, then $L(s,\mathrm{Sym}^n\pi,K)\in\mathcal{A}_{n+1}(K)$ for $n\leq 8$.  Despite this recent progress, because of our limited knowledge of automorphy, the number of symmetric powers needed to access the interval $I$ is particularly important in the sorts of analytic problems considered in this paper.

Recall that the Chebyshev polynomials $U_n(t)$, defined by
\[
\sum_{n=0}^\infty U_n(t) x^n = \frac{1}{1-2tx+x^2},
\]
form an orthnormal basis for $L^2([-1,1],\mu_{\mathrm{ST}})$.  If $\pi_\mathfrak{p}$ is unramified, then $U_n(\cos\theta_{\p})$ is the Dirichlet coefficient of $L(s,\mathrm{Sym}^n \pi,K)$ at the prime $\mathfrak{p}$.  We say that a subset $I\subseteq[-1,1]$ can be {\bf $\mathrm{\mathbf{Sym}^N}$-minorized} if there exist constants $b_0,\dots, b_N\in\mathbb{R}$ with $b_0>0$ such that
\begin{equation}
\label{eqn:symn-minor1}
{\bf 1}_{I}(t) \geq \sum_{n=0}^N b_n U_n(t)
\end{equation}
for all $t\in[-1,1]$, where ${\bf 1}_{I}(\cdot)$ denotes the indicator function of $I$.  Note that if $I$ can be $\mathrm{Sym}^N$-minorized, then it is the union of intervals which individually need not be $\mathrm{Sym}^N$-minorizable.  We prove the following result.

\begin{theorem} \label{thm:hoheisel_sato_tate}
Assume the above notation.  Let $K$ be a totally real number field, and let $\pi\in\mathcal{A}_2(K)$ be a non-CM genuine representation which satisfies GRC and has trivial central character.  Suppose that a fixed subset $I\subseteq[-1,1]$ can be $\mathrm{Sym}^N$-minorized and that $\mathrm{Sym}^{n} \pi\in\mathcal{A}_{n+1}(K)$ for each $n\leq N$.  Let $B=\max_{0\leq n\leq N}|b_n|/b_0$, where $b_0,\ldots,b_N$ are as in \eqref{eqn:symn-minor1}.  There exists an absolute constant $\Cl[abcon]{satotatec1}>0$ such that if
\[
\delta \leq \frac{\Cr{satotatec1}}{N^4 [K:\Q]\log(3BN)},
\]
$x$ is sufficiently large, and $h\geq x^{1-\delta}$, then
\[
\sum_{\substack{x<\mathrm{N}\p\leq x+h \\ \textup{$\pi_{\p}$ unramified}}} {\bf 1}_{I}(\cos \theta_{\p}) \log\mathrm{N}\p\asymp h,
\]
where the implied constant depends on $B$, $I$, and $K$.  In particular, if $I$ can be $\mathrm{Sym}^4$-minorized, or if $I$ can be $\mathrm{Sym}^8$-minorized and $\pi$ is a Hecke newform over $\mathbb{Q}$, then this is unconditional.\end{theorem}

\begin{remark1}
1. For any fixed $n$, determining the intervals $I$ that can be $\mathrm{Sym}^N$-minorized is an elementary combinatorial problem.  We carry this out in Lemma \ref{lem:sym4-minorized} to determine the intervals that can be $\mathrm{Sym}^4$-minorized, which we consider to be the most interesting case; it turns out that the proportion of subintervals of $[-1,1]$ which can be $\mathrm{Sym}^4$-minorized is roughly $0.388$.  If one is not concerned with obtaining the optimal minorant or if $N$ is large, it is likely more convenient to apply a standard minorant for $I$ instead.  For the Beurling-Selberg minorant (see Montgomery \cite[Lecture 1]{Montgomery10}), a tedious calculation shows that if $N\geq4(1+\delta)/\mu_{\mathrm{ST}}(I)-1$ for some $\delta>0$, then $I$ can be $\mathrm{Sym}^N$-minorized with
\[
B\leq \frac{2+3/\delta}{\mu_{\mathrm{ST}}(I)}.
\]
It follows that any interval can be $\mathrm{Sym}^N$-minorized for $N$ sufficiently large, and thus every interval is at least conditionally covered by Theorem \ref{thm:hoheisel_sato_tate}; Lemma \ref{lem:extreme-values} shows, however, that this minorant might be far from optimal.  With the Beurling-Selberg minorant, we prove unconditional results for intervals $I$ satisfying $\mu_{\mathrm{ST}}(I)>\frac{4}{5}$.  By contrast, Lemma \ref{lem:sym4-minorized} implies unconditional results for all intervals satisfying $\mu_{\mathrm{ST}}(I) \geq 0.534$, and for some with measure as small as $0.139$.

2. It is tempting to ask whether one can exploit existing results on potential automorphy for symmetric power $L$-functions and the explicit dependence on the base field in Theorem \ref{main-theorem} to obtain unconditional, albeit weaker, results for all subintervals of $[-1,1]$.  The proof of the Sato-Tate conjecture crucially relies on the work of Moret-Bailly \cite{Moret-Bailly} establishing the existence of number fields over which certain varieties have points.  The proof of this result unfortunately only permits control over the ramification at finitely many places, so it is not possible to even obtain bounds on the discriminants of the fields over which the symmetric power $L$-functions are automorphic.  Thus, the authors do not believe it is possible to obtain an unconditional analogue of Theorem \ref{thm:hoheisel_sato_tate} for all $I$ at this time.
\end{remark1}

As mentioned earlier, Theorem \ref{main-theorem} also allows us to access Linnik-type questions.  As one such example, we consider an analogue of Linnik's theorem in the context of the Sato-Tate conjecture.  One complication in the proof of Linnik's theorem that is not seen in Hoheisel's is the possible existence of a so-called Siegel zero for some Dirichlet $L$-function $L(s,\chi)$.  In order to handle this possible contribution (as one must, since Linnik's theorem is unconditional), two facts are used: there is at most one character $\chi\pmod{q}$ whose associated $L$-function has a Siegel zero, and every coefficient in the $\pmod{q}$ Fourier decomposition of the indicator function of set $\{n\in\mathbb{Z}\colon n\equiv a\pmod{q}\}$ is of the same size.  Neither of these facts need be true for symmetric power $L$-functions $L(s,\mathrm{Sym}^n\pi,K)$ and the minorant \eqref{eqn:symn-minor1}, so we consequently say that the minorant \eqref{eqn:symn-minor1} {\bf does not admit Siegel zeros} if for every $1\leq n\leq N$ for which $L(s,\mathrm{Sym}^n \pi,K)$ has a Siegel zero, the coefficient $b_n$ satisfies $b_n\leq 0$.  (It happens that if $b_n\leq 0$, then the Siegel zero may be trivially ignored in the analysis.)  Finally, if a set $I\subseteq [-1,1]$ admits such a minorant, then we say that $I$ can be $\mathrm{Sym}^N$-minorized without admitting Siegel zeros.  We have suppressed the role of the representation $\pi$ in this terminology, but its presence will always be clear in context. 

\begin{theorem}
\label{thm:least_sato_tate}
Assume the notation of Theorem \ref{thm:hoheisel_sato_tate}, and in particular that $I\subset[-1,1]$ can be $\mathrm{Sym}^N$-minorized.  Let $\pi\in\mathcal{A}_2(K)$ satisfy the hypotheses of Theorem \ref{thm:hoheisel_sato_tate}.  Suppose further that the $\mathrm{Sym}^N$-minorant admits no Siegel zeros.  If the Dedekind zeta function $\zeta_K(s)$ has no Siegel zero, then there exists an absolute constant $\Cl[abcon]{satotatec2}>0$ such that if $\mathrm{Sym}^n \pi \in\mathcal{A}_{n+1}(K)$ for $n\leq N$, then there is an unramified prime $\mathfrak{p}$ satisfying both $\cos\theta_{\p} \in I$ and
\[
\mathrm{N}\p \ll \max\Big\{N^N\mathfrak{q}(\pi)^{N^3},D_K [K:\Q]^{[K:\Q]}\Big\}^{\Cr{satotatec2}N^4 \log(3BN)}.
\]
If $\pi$ is associated to a non-CM newform over $\mathbb{Q}$ with squarefree level, then this may be improved to
\[
\mathrm{N}\p \ll (N \mathfrak{q}(\pi))^{\Cr{satotatec2}N^5 \log(3BN)}.
\]
\end{theorem}

\begin{remark1}
1.  Even if $\zeta_K(s)$ has a Siegel zero, we can still prove an effective bound for the least norm of a prime ideal in the Sato-Tate conjecture, but the bound will have a less desirable dependence on $K$.  See the remark following the proof of Theorem \ref{thm:least_sato_tate}.

2.  When $I$ is fixed and $\pi$ varies, the bound in Theorem \ref{thm:least_sato_tate} has the shape $\mathrm{N}\mathfrak{p} \leq \mathfrak{q}(\pi)^A$ for some absolute constant $A$, and so is comparable to Linnik's theorem.  However, if $\pi$ is fixed and $I$ is varying, the dependence is much worse.  This comes partially from the constants in the zero-free region for $L(s,\mathrm{Sym}^n\pi,K)$, where the $n$ dependence in particular is of the form $n^4$.  Without improving the quality of these constants, it seems likely that only minor improvements can be made to Theorem \ref{thm:least_sato_tate}.

3.  Suppose that $\pi\in\mathcal{A}_2(K)$ is self-dual.  It follows from work of Hoffstein and Ramakrishnan \cite{Hoffstein} that neither $L(s,\pi,K)$ nor $L(s,\mathrm{Sym}^2\pi,K)$ has a Siegel zero.  In fact, their proof of Theorem B also shows that if $L(s,\mathrm{Sym}^j \pi, K)$ is automorphic for $j \in \{n-2,n,n+2\}$, then $L(s,\mathrm{Sym}^n \pi, K)$ does not have a Siegel zero.  From the known automorphy results mentioned earlier, it follows that $L(s,\mathrm{Sym}^n\pi,K)$ does not have a Siegel zero for $n=1$ and $2$, and additionally $n=3$, $4$, $5$, and $6$ if $K \cap \mathbb{Q}(e^{2\pi i/35}) = \mathbb{Q}$. 

4.  Following the ideas of Moreno \cite[Theorem 4.2]{Moreno2}, we could prove a version of the zero repulsion phenomenon of Deuring and Heilbronn for $L(s,\pi\otimes\pi',K)$.  Such a result would allow us to weaken the definition of $I$ not admitting Siegel zeros.  In particular, we would only need to require that for every $1\leq n\leq N$ such that $L(s,\mathrm{Sym}^n\pi, K)$ has a Siegel zero, the coefficient $b_n$ satisfies $b_n\leq b_0$.  Since this does not completely eliminate the Siegel zero contribution, we do not carry out this computation.

5.  If $K=\Q$, one may use Corollary \ref{cor:lfzde2} instead of Corollary \ref{cor:lfzde} in the proof of Theorem \ref{thm:least_sato_tate}.  This would produce a bound on the least prime $p\equiv a\pmod q$ such that $\cos\theta_p\in I$.
\end{remark1}

This paper is organized as follows.  In Section \ref{sec:definitions}, we discuss the basic properties of automorphic $L$-functions that we will use in the proofs of the theorems; we also prove a few useful lemmas.  In Section \ref{sec:proof-main}, we prove Theorems \ref{main-theorem} and \ref{main-theorem-2}.  In Section \ref{sec:applications}, we prove Theorems \ref{effective-hoheisel}, \ref{thm:hoheisel_sato_tate}, and \ref{thm:least_sato_tate}. In Section \ref{sec:auto_BV}, we prove Theorem \ref{thm:auto_BV}.

\section*{Acknowledgements}
\noindent
This research began while the authors visited the Centre de Recherches Math\'ematiques for the analytic semester of the thematic year on number theory.  The authors are grateful for this opportunity and would particularly like to thank Andrew Granville, Chantal David, and Dimitris Koukoulopoulos for their generous hospitality.  We also thank Valentin Blomer for his comments and for bringing the work of Motohashi to our attention, Sary Drappeau for bringing the work of Clozel and Thorne to our attention, and the anonymous referee for carefully reading this paper and providing many useful and insightful comments.

\section{Preliminaries}
\label{sec:definitions}

\subsection{Definitions and notation}
\label{subsec:definitions}

We follow the account of Rankin-Selberg $L$-functions given by Brumley \cite[Section 1]{Brumley}.  Let $K/\Q$ be a number field of absolute discriminant $D_K$.  Let $\mathbb{A}_K$ denote the ring of adeles over $K$, and let $\mathcal{A}_d(K)$ be the set of cuspidal automorphic representations of $\mathrm{GL}_d(\mathbb{A}_K)$ with unitary central character.

Let $\pi\in\mathcal{A}_d(K)$.  We have the factorization $\pi=\otimes_{\mathfrak{v}}\pi_{\mathfrak{v}}$ over the places of $K$.  For each nonarchimedean $\mathfrak{p}$, we have the Euler factor
\begin{equation*}
L_{\p}(s,\pi,K)=\prod_{j=1}^d(1-\alpha_{\pi}(j,\p)\mathrm{N}\p^{-s})^{-1}
\end{equation*}
associated with $\pi_\p$.  Let $R_\pi$ be the set of prime ideals $\p$ for which $\pi_\p$ is ramified.  We call $\alpha_{\pi}(j,\p)$ the local roots of $L(s,\pi,K)$ at $\p$, and if $\p\notin R_{\pi}$, then $\alpha_{\pi}(j,\p)\neq0$ for all $1\leq j\leq d$.  The representation $\pi$ has an associated automorphic $L$-function whose Euler product and Dirichlet series are given by
\begin{equation*}
L(s,\pi,K)=\prod_{\p}L_{\p}(s,\pi,K)=\sum_{\a}\frac{\lambda_{\pi}(\a)}{\mathrm{N}\a^{s}},
\end{equation*}
where $\p$ runs through the finite primes and $\a$ runs through the non-zero integral ideals of $K$.  This Euler product converges absolutely for $\re(s)>1$, which implies that $|\alpha_{\pi}(j,\p)|<\mathrm{N}\p$.  Luo, Rudnick, and Sarnak \cite[Theorem 2]{LRS} showed that if $\p\notin R_{\pi}$, then
\begin{equation}
\label{eqn:LRS}
|\alpha_{\pi}(j,\p)|\leq\mathrm{N}\p^{\frac{1}{2}-\frac{1}{d^2+1}},
\end{equation}
and M{\"u}ller and Speh \cite{MS} proved that this holds for all primes $\p$.  The generalized Ramanujan conjecture (GRC) asserts a further improvement.

\begin{hypothesis*}[GRC]
\label{GRC}
Assume the above notation.  For each prime $\p\notin R_{\pi}$, we have $|\alpha_{\pi}(j,\p)|=1$, and for each prime $\p\in R_{\pi}$, we have $|\alpha_{\pi}(j,\p)|\leq 1$.
\end{hypothesis*}
\begin{remark}
It is expected that all automorphic $L$-functions $L(s,\pi,K)$ satisfy GRC.  Indeed, it is already known for many of the most commonly used automorphic $L$-functions.  Such $L$-functions include Hecke $L$-functions and the $L$-function of a cuspidal normalized Hecke eigenform of positive even integer weight $k$ on the congruence subgroup $\Gamma_0(N)$.
\end{remark}

At each archimedean place $\v$, we associate to $\pi_\v$ a set of $n$ complex numbers $\{\mu_{\pi}(j,\v)\}_{j=1}^{d}$, often called Langlands parameters, which are known to satisfy
\[
\re(\mu_\pi(j,\v))>-\frac{1}{2}+\frac{1}{d^2+1}
\]
by the work of Luo, Rudnick, and Sarnak \cite{LRS}.  The local factor at $\v$ is defined to be
\[
L_{\v}(s,\pi,K)=\prod_{j=1}^{d}\Gamma_{K_{\v}}(s+\mu_{\pi}(j,\v)),
\]
where $\Gamma_{\R}(s)=\pi^{-s/2}\Gamma(\frac{s}{2})$ and $\Gamma_{\mathbb{C}}(s)=\Gamma_{\R}(s)\Gamma_{\R}(s+1)$. Letting $S_{\infty}$ denote the set of archimedean places, we define the gamma factor of $L(s,\pi,K)$ by
\[
L_{\infty}(s,\pi,K)=\prod_{\v\in S_{\infty}}L_{\v}(s,\pi,K).
\]
For notational convenience, we will define the complex numbers $\kappa_{\pi}(j)$ by
\begin{equation}
\label{eqn:kappa}
	L_{\infty}(s,\pi,K)=\prod_{j=1}^{d [K:\Q]}\Gamma_{\R}(s+\kappa_{\pi}(j)).
\end{equation}

Any automorphic $L$-function $L(s,\pi,K)$ admits a meromorphic continuation to $\mathbb{C}$ with poles possible only at $s=0$ and $1$. Let $r(\pi)$ denote the order of the pole at $s=1$, and define the completed $L$-function
\begin{equation*}
\Lambda(s,\pi,K)=(s(1-s))^{r(\pi)} q(\pi)^{s/2}L_{\infty}(s,\pi,K)L(s,\pi,K),
\end{equation*}
where $q(\pi)$ is the conductor of $\pi$.  (Note that $D_K^d$ divides $q(\pi)$.)  It is well-known that $\Lambda(s,\pi,K)$ is an entire function of order 1 and that there exists a complex number $\varepsilon(\pi)$ of modulus 1 such that $\Lambda(s,\pi,K)$ satisfies the functional equation
\begin{equation*}
\Lambda(s,\pi,K)=\varepsilon(\pi)\Lambda(1-s,\tilde{\pi},K),
\end{equation*}
where $\tilde{\pi}$ is the representation contragredient to $\pi$.  For each $\mathfrak{p}$, we have that $\{\alpha_{\tilde{\pi}}(j,\p)\colon1\leq j\leq d\}=\{\overline{\alpha_{\pi}(j,\p)}\colon 1\leq j\leq d\}$.  Moreover, 
\begin{equation*}
L_{\infty}(s,\tilde{\pi},K)=L_{\infty}(s,\pi,K)~\text{ and }~ q(\tilde{\pi})=q(\pi).
\end{equation*}

To maintain uniform estimates for the analytic quantities associated to $L(s,\pi,K)$, we define the {\bf analytic conductor} of $L(s,\pi,K)$ by
\begin{equation}
\label{eqn:analytic_conductor}
\q(s,\pi)=q(\pi)\prod_{j=1}^{d [K:\Q]}(|s+\kappa_{\pi}(j)|+3).
\end{equation}
We will frequently make use of the quantity $\q(0,\pi)$, which we denote by $\q(\pi)$.

As in the introduction, define the von Mangoldt function $\Lambda_\pi(\a)$ by
\begin{equation*}
-\frac{L'}{L}(s,\pi,K)=\sum_{\a}\frac{\Lambda_{\pi}(\a)}{\mathrm{N}\a^{s}},
\end{equation*}
and let $\Lambda_K(\a)$ be that associated to the Dedekind zeta function $\zeta_K(s)$.  We then have that
\[
\Lambda_{\pi}(\a)=\lambda_{\pi}(\a)\Lambda_K(\a).
\]
Using the bounds for $|\alpha_{\pi}(j,\p)|$ from \cite{LRS,MS}, we have that
\begin{equation}
\label{eqn:almost-GRC}
|\Lambda_{\pi}(\a)|\leq d\Lambda_K(\a)\mathrm{N}\a^{\frac{1}{2}-\frac{1}{d^2+1}}
\end{equation}
for every ideal $\a$, and under GRC, we have $|\Lambda_{\pi}(\a)|\leq d\Lambda_K(\a)$. 

Consider two representations $\pi\in\mathcal{A}_d(K)$ and $\pi'\in\mathcal{A}_{d'}(K)$.  We are interested in the Rankin-Selberg product $\pi\otimes \pi^\prime$ of $\pi$ and $\pi^\prime$, which, at primes $\mathfrak{p}\notin R_\pi \cup R_{\pi^\prime}$, has a local factor given by
\[
L_{\p}(s,\pi\otimes\pi',K) = \prod_{j_1=1}^{d}\prod_{j_2=1}^{d'}(1-\alpha_{\pi}(j_1,\p)\alpha_{\pi'}(j_2,\p)\mathrm{N}\p^{-s})^{-1}.
\]
For $\p\in R_{\pi}\cup R_{\pi'}$, we write the local roots as $\beta_{\pi\otimes\pi'}(j,\p)$ with $1\leq j\leq d'd$, and for each such $\p$ we define
\[
L_{\p}(s,\pi\otimes\pi',K)=\prod_{j=1}^{d'd}(1-\beta_{\pi\otimes\pi'}(j,\p)\mathrm{N}\p^{-s})^{-1}.
\]

This gives rise to the Rankin-Selberg $L$-function $L(s,\pi\otimes\pi',K)$, whose Euler product and gamma factor are given by
\[
L(s,\pi\otimes\pi',K)=\prod_{\p}L_{\p}(s,\pi\otimes\pi',K)
\]
and
\[
L_{\infty}(s,\pi\otimes\pi',K)=\prod_{\v\in S_{\infty}}\prod_{j_1=1}^{d}\prod_{j_2=1}^{d'}\Gamma_{K_{\v}}(s+\mu_{\pi\otimes\pi'}(j_1,j_2,\v))=\prod_{j=1}^{d'd [K:\Q]}\Gamma_{\R}(s+\kappa_{\pi\otimes\pi'}(j)),
\]
where (again by \cite{LRS})
\begin{equation}
\label{eqn:selberg}
\re(\kappa_{\pi\otimes\pi'}(j))\geq-1+\frac{1}{d^2+1}+\frac{1}{(d^\prime)^2+1}.
\end{equation}

\begin{remark}
It is possible that there may be trivial zeros of $L(s,\pi\otimes\pi^\prime,K)$ inside the critical strip, which would arise from poles of $L_{\infty}(s,\pi\otimes\pi',K)$.  However, it follows from \eqref{eqn:selberg} that Theorem \ref{main-theorem} accounts for these zeros.
\end{remark}

By \cite[Equation 8]{Brumley}, we have
\begin{equation}
\label{eqn:ac-ineq}
\q(s,\pi\otimes\pi')\leq\q(\pi)^{d'}\q(\pi')^d (|s|+3)^{d'd [K:\Q]}.
\end{equation}
Finally, we note that if $\pi'\cong\tilde{\pi}$, then the order $r(\pi\otimes\pi^\prime)$ of the pole at $s=1$ of $L(s,\pi\otimes\pi',K)$ is $1$.

\subsection{Preliminary lemmas}
We begin with a zero-free region for $L(s,\pi\otimes\pi',K)$, obtained by adapting Theorem 5.10 of Iwaniec and Kowalski \cite{IK} to $L$-functions over arbitrary number fields.  Recall the notation $\mathcal{D}$ and $\mathcal{Q}$ introduced in Theorem \ref{main-theorem}.

\begin{lemma}
\label{1.9}
Suppose either that at least one of $d$ and $d'$ equals one or that at least one of $\pi$ and $\pi'$ is self-dual.  Let $T\geq 1$, and let
\begin{equation}
\label{eqn:script_L}
\mathcal{L}=\mathcal{L}(T,\pi\otimes\pi',K)= \mathcal{D}\log(\mathcal{Q}T^{[K:\Q]}).
\end{equation}
There is a positive absolute constant $\Cl[abcon]{zero-free}$\footnote{We denote by $c_1,c_2,\dots$ a sequence of constants, each of which is absolute, positive, and effectively computable.  We do not recall this convention in future statements, as we find it to be notationally cumbersome.} such that the region
\[
\{s=\sigma+it\colon\sigma\geq 1-\Cr{zero-free}\mathcal{L}^{-1},|t|\leq T\}
\]
contains at most one zero of $L(s,\pi\otimes\pi',K)$.  If such an exceptional zero $\beta_{\mathrm{exc}}$ exists, then it is real and simple, $L(s,\pi\otimes\pi',K)$ must be self-dual, and $\beta_{\mathrm{exc}}<1$.  We call such an exceptional zero $\beta_{\mathrm{exc}}$ a Siegel zero.
\end{lemma}
\begin{remark}
	Note that from the definition of $\mathcal{L}$, we always have $\mathcal{L}\gg d'd[K:\Q]$.
\end{remark}
\begin{proof}
With the help of \eqref{eqn:ac-ineq}, the proofs are the same as those in \cite[Section 5.4]{IK} and \cite[Section 3]{Moreno2}.  See also \cite{MR2078295} and the sources contained therein.  
%
%
\end{proof}

\begin{lemma}
\label{1.10}
Let $T\geq0$, and let $\tau\in\R$ satisfy $|\tau|\leq T$.
\begin{enumerate}
\item Uniformly on the disk $|s-(1+i\tau)|\leq1/4$, we have that
\[
\frac{L'}{L}(s,\pi\otimes\pi',K)+\frac{r(\pi\otimes\pi')}{s}+\frac{r(\pi\otimes\pi')}{s-1}-\sum_{|\rho-(1+i\tau)|\leq1/2}\frac{1}{s-\rho}\ll \mathcal{L},
\]
where the sum runs over zeros $\rho$ of $L_S(s,\pi\otimes\pi^\prime,K)$.
\item For $1\geq \eta \gg \mathcal{L}^{-1}$, we have that
\[
\sum_{|\rho-(1+i\tau)|\leq \eta}1\ll \eta\mathcal{L},
\]
where the sum runs over zeros $\rho$ of $L_S(s,\pi\otimes\pi^\prime,K)$.
\end{enumerate}
\end{lemma}
\begin{proof}
Part 1 is Lemma 2.4 of Akbary and Trudgian \cite{AT}.  Part 2 follows from combining Theorem 5.6 of \cite{IK} and Proposition 5.8 of \cite{IK}.
\end{proof}

Let $ S  =  S (\pi \otimes \pi^\prime) = R_\pi \cup R_{\pi^\prime}$, and define the partial $L$-function
\[
L_S(s,\pi\otimes\pi^\prime,K) = \prod_{\p \not\in S} L_\p(s,\pi\otimes\pi^\prime,K).
\]
If $\a$ is coprime to every $\p \in S$, then $\lambda_{\pi\otimes \pi^\prime}(\a) = \lambda_\pi(\a) \lambda_{\pi^\prime}(\a)$.  Thus, the partial $L$-function $L_S(s,\pi \otimes \pi^\prime, K)$ can be directly related to $L(s,\pi,K)$ and $L(s,\pi^\prime,K)$.  Moreover, it is apparent that if $\sigma > 1/2$, then $N^S_{\pi\otimes\pi^\prime}(\sigma,T) = N_{\pi\otimes\pi^\prime}(\sigma,T)$, where
\[
N^S_{\pi\otimes\pi^\prime}(\sigma,T) = \#\{\rho=\beta+i\gamma\colon L_S(\pi,\pi\otimes\pi^\prime,K)=0, \beta\geq\sigma, |\gamma|\leq T\}.
\]
Thus, it suffices to prove Theorem \ref{main-theorem} and its variants for the partial $L$-function.  Finally, note that both Lemma \ref{1.9} and \ref{1.10} hold for $L_S(s,\pi\otimes\pi^\prime,K)$.  For convenience, we write $(\a, S )=1$ if the ideal $\a$ has no prime factors in $S$.

\begin{lemma}
\label{1.11}
If $\eta>0$ and $y\geq 2$, then
\begin{enumerate}
\item $\displaystyle \sum_{(\a, S )=1}\frac{|\Lambda_{\pi\otimes\pi'}(\a)|}{\mathrm{N}\a^{1+\eta}}\ll \frac{1}{\eta}+d'd\log(\q(\pi)\q(\pi')).$
\item $\displaystyle \sum_{\substack{\mathrm{N}\a\leq y \\ (\a, S )=1}}\frac{|\Lambda_{\pi\otimes\pi'}(\a)|}{\mathrm{N}\a}\ll \log y+d'd\log(\q(\pi)\q(\pi'))$.
\end{enumerate}
\end{lemma}

\begin{proof}
By the Cauchy-Schwarz inequality, we have
\begin{align*}
\sum_{(\a, S )=1}\frac{|\Lambda_{\pi\otimes\pi'}(\a)|}{(\mathrm{N}\a)^{1+\eta}}&\ll\Big(-\frac{L'}{L}(1+\eta,\pi\otimes\tilde{\pi},K)\Big)^{1/2}\Big(-\frac{L'}{L}(1+\eta,\pi'\otimes\tilde{\pi}',K)\Big)^{1/2}.
\end{align*}

We first estimate $-\frac{L'}{L}(1+\eta,\pi\otimes\tilde{\pi},K)$, which is a positive quantity because $\eta>0$ and the Dirichlet coefficients of $-\frac{L'}{L}(s,\pi\otimes\tilde{\pi},K)$ are real and nonnegative.  By Theorem 5.6 of \cite{IK} and part 3 of Proposition 5.7 of \cite{IK}, we have
\begin{align*}
-\re\Big(\frac{L'}{L}(1+\eta,\pi\otimes\tilde{\pi},K)\Big)=&\frac{1}{2}\log q(\pi\otimes\tilde{\pi})+\re\Big(\frac{L_{\infty}'}{L_{\infty}}(1+\eta,\pi\otimes\tilde{\pi},K)\Big)\\
&+\frac{1}{1+\eta}+\frac{1}{\eta}-\sum_{\rho\neq0,1}\re\Big(\frac{1}{1+\eta-\rho}\Big),
\end{align*}
where $\rho=\beta+i\gamma$ runs through the zeros of $\Lambda(s,\pi\otimes\pi',K)$.  Since $0\leq \beta<1$, we have
\[
\re\Big(\frac{1}{1+\eta-\rho}\Big)=\frac{1+\eta+\beta}{(1+\eta+\beta)^2+\gamma^2} > 0,
\]
the contribution from sum over zeros is negative, so we can discard it.  Thus
\[
-\re\Big(\frac{L'}{L}(1+\eta,\pi\otimes\tilde{\pi},K)\Big)\leq\frac{1}{2}\log q(\pi\otimes\tilde{\pi})+\re\Big(\frac{L_{\infty}'}{L_{\infty}}(1+\eta,\pi\otimes\tilde{\pi},K)\Big)+\frac{1}{1+\eta}+\frac{1}{\eta}.
\]
By the proof of part 2 in Proposition 5.7 in \cite{IK}, we have
\[
\re\Big(\frac{L_{\infty}'}{L_{\infty}}(1+\eta,\pi\otimes\tilde{\pi},K)\Big)=-\sum_{|s+\kappa_{\pi\otimes\tilde{\pi}}(j)|<1}\re\Big(\frac{1}{1+\eta+\kappa_{\pi\otimes\tilde{\pi}}(j)}\Big)+O(\log\q(\pi\otimes\tilde{\pi})).
\]
Since \eqref{eqn:selberg} holds for all $1\leq j\leq d'd [K:\Q]$, it follows that
\[
\re\Big(\frac{1}{1+\eta+\kappa_{\pi\otimes\tilde{\pi}}(j)}\Big) \geq \frac{\eta}{\eta^2+\im(\kappa_{\pi\otimes\tilde{\pi}}(j))^2} > 0.
\]
Therefore, by positivity and \eqref{eqn:ac-ineq},
\[
-\frac{L'}{L}(1+\eta,\pi\otimes\tilde{\pi},K)\ll\frac{1}{\eta}+\log(\q(\pi\otimes\tilde{\pi}))\ll\frac{1}{\eta}+d\log\q(\pi).
\]
Since the analogue must hold for $\pi'$, part 1 follows.  Part 2 follows by choosing $\eta=\frac{1}{\log y}.$
\end{proof}

We conclude this section with a bound on the mean value of a Dirichlet polynomial.

\begin{lemma}
\label{main-inequality}
Let $T\geq 1$ and $u>y\gg (\mathcal{Q}T^{[K:\Q]})^{\Cl[abcon]{yexp1}}$, where $\Cr{yexp1}$ is sufficiently large.  Define
\[
S_{y,u}(\tau,\pi\otimes\pi'):=\sum_{y<\mathrm{N}\p\leq u}\frac{\Lambda_{\pi\otimes\pi^\prime}(\p)}{\mathrm{N}\p^{1+i\tau}}.
\]
1. If $L(s,\pi',K)$ satisfies GRC, then
\[
\int_{-T}^T |S_{y,u}(\tau,\pi\otimes\pi')|^2 d\tau\ll \frac{(d^\prime)^2(\log u)(\log u+d^2\log\mathfrak{q}(\pi))}{\log y}.
\]
2. If $K=\Q$ and $L(s,\pi,\Q)$ satisfies GRC, then
\[
\sum_{q\leq T^2}\log\frac{T^2}{q}~\sideset{}{^\star}\sum_{\chi\bmod q}\int_{-T}^{T}\Big|S_{y,u}(\tau,(\pi\otimes\chi)\otimes\pi')\Big|^2 d\tau\ll (d^\prime)^2(\log u)(\log u+d^2\log\mathfrak{q}(\pi)).
\]
\end{lemma}
\begin{proof}
1. Let $b(\p)$ be a complex-valued function supported on the prime ideals of $K$ such that $\sum_{\p}\mathrm{N}\p|b(\p)|^2<\infty$ and $b(\p)=0$ whenever $\mathrm{N}\p\leq y$.  With our choice of $T$ and $y$, it follows from \cite[Corollary 3.8]{weiss} that
\begin{equation}
\label{eqn:large_sieve_ineq}
\int_{-T}^{T}\Big|\sum_{\p}b(\p)\mathrm{N}\p^{-i\tau}\Big|^2 d\tau\ll \frac{1}{\log y}\sum_{\p}\mathrm{N}\p|b(\p)|^2,
\end{equation}
If we define $b(\p)$ by
\[
b(\p) =\begin{cases}
\displaystyle\frac{\Lambda_{\pi\otimes \pi^\prime}(\mathfrak{p})}{\mathrm{N}\mathfrak{p}}&\mbox{if $y<\mathrm{N}\p\leq u$,}\\
0&\mbox{otherwise}
\end{cases}
\]
and recall the definition of $S_{y,u}(\tau,\pi\otimes\pi')$, then an application of \eqref{eqn:large_sieve_ineq} yields the bound
\[
\int_{-T}^{T}\Big|S_{y,u}(\tau,\pi\otimes\pi')\Big|^2 d\tau\ll \frac{1}{\log y}\sum_{y<\mathrm{N}\p\leq u}\frac{|\Lambda_{\pi\otimes\pi'}(\p)|^2}{\mathrm{N}\p}.
\]
Since $y$ is greater than the norm of any ramified prime, it follows from our assumption of GRC for $L(s,\pi^\prime,K)$ that
\[
\sum_{y<\mathrm{N}\p\leq u}\frac{|\Lambda_{\pi\otimes\pi'}(\p)|^2}{\mathrm{N}\p}=\sum_{y<\mathrm{N}\p\leq u}\frac{|\lambda_{\pi}(\p)|^2|\lambda_{\pi'}(\p)|^2(\log\mathrm{N}\p)^2}{\mathrm{N}\p}\leq (d^\prime)^2\sum_{y<\mathrm{N}\p\leq u}\frac{|\lambda_{\pi}(\p)|^2(\log\mathrm{N}\p)^2}{\mathrm{N}\p}.
\]
Since all prime ideals $\p$ in the sum are unramified, we have that $|\lambda_{\pi}(\p)|^2\log\mathrm{N}\p=|\Lambda_{\pi\otimes\tilde{\pi}}(\p)|$.  The claimed result now follows by partial summation using Lemma \ref{1.11}.

2. Let $K=\Q$.  Suppose that $a(p)$ is a function on primes such that $a(p)=0$ if $p\leq Q$ and $\sum_{p}p|a(p)|^2<\infty$.  By \cite[Theorem 4]{Gallagher2}, we have that for $T\geq 1$, 
\begin{equation}
\label{eqn:sum_unram}
\sum_{q\leq Q}\log\frac{Q}{q}~\sideset{}{^\star}\sum_{\chi\bmod q}\int_{-T}^{T}\Big|\sum_{p}a(p)\chi(p)p^{-it}\Big|^2 dt\ll \sum_{p}(Q^2 T+p)|a(p)|^2.
\end{equation}
Define $a(p)$ as we did $b(\p)$ above and let $Q=T^2$.  Note that our choice of $y$ implies that $a(p)=0$ at every ramified prime $p$ dividing the conductor of $\pi\otimes\pi'$.  Choosing $\Cr{yexp1}>6$, our hypotheses imply that $T^5\ll p$ for every $p\in(u,y]$.  Thus
\[
\sum_{q\leq T^2}\log\frac{T^2}{q}~\sideset{}{^\star}\sum_{\chi\bmod q}\int_{-T}^{T}\Big|S_{y,u}(\tau,(\pi\otimes\chi)\otimes\pi')\Big|^2 dt\ll \sum_{y<p\leq u}(T^5+p)\frac{|\Lambda_{\pi\otimes\pi'}(p)|^2}{p^2}\leq \sum_{y<p\leq u}\frac{|\Lambda_{\pi\otimes\pi'}(p)|^2}{p}.
\]
This is bounded using GRC just as in the proof of Part 1.
\end{proof}

\section{The zero density estimate}
\label{sec:proof-main}

In this section, we prove Theorem \ref{main-theorem} by generalizing Gallagher's \cite{Gallagher2} and Weiss's \cite{weiss} treatment of Tur\'an's method for detecting zeros of $L$-functions, obtaining a result that is uniform in  $K$, $\pi$, and $\pi'$.  The key result is the following technical proposition, whose proof we defer to the end of the section.  Recall from Lemma \ref{1.9} that $\mathcal{L}=\mathcal{D}\log(\mathcal{Q}T^{[K:\Q]})$, and any real zero $\beta_{\mathrm{exc}}$ of $L(s,\pi\otimes\pi',K)$ to the right of the region $\sigma\leq1-\Cr{zero-free}\mathcal{L}^{-1}$ is called a Siegel zero.

\begin{proposition}
\label{4.2}
Recall the notation and hypotheses of Theorem \ref{main-theorem}.  Let $y=e^{\Cr{yexp1}\mathcal{L}}$ with $\Cr{yexp1}$ sufficiently large.  Suppose that $\eta$ satisfies $\mathcal{L}^{-1} \ll \eta \leq 1/55$.  Let
\[
S_{y,u}(\tau,\pi\otimes\pi'):=\sum_{y<\mathrm{N}\p\leq u}\frac{\Lambda_{\pi\otimes\pi^\prime}(\p)}{\mathrm{N}\p^{1+i\tau}}.
\]
If $L_S(s,\pi\otimes\pi^\prime)$ has a zero $\rho_0$ satisfying $|\rho_0-(1+i\tau)|\leq \eta$ and $\rho_0$ is not a Siegel zero, then for sufficiently large $\Cl[abcon]{yexp2}$ and $\Cl[abcon]{yexp3}$, we have that
\[
\frac{y^{\Cr{yexp2} \eta}}{(\log y)^3}\int_y^{y^{\Cr{yexp3}}}|S_{y,u}(\tau,\pi\otimes\pi')|^2\frac{du}{u}\gg1.
\]
\end{proposition}

We first deduce Theorem \ref{main-theorem} from Proposition \ref{4.2}.  The proof of Proposition \ref{4.2} relies on certain upper and lower bounds on the derivatives of $\frac{L_S^\prime}{L_S}(s,\pi\otimes\pi^\prime)$, which are proven and assembled subsequently.

\subsection{Proof of Theorems \ref{main-theorem} and \ref{main-theorem-2}}

By Theorem 5.8 of \cite{IK}, we have
\begin{equation}
\label{eqn:zero-count}
N_{\pi\otimes\pi'}(0,T)=\frac{T}{\pi}\log\Big(q(\pi\otimes\pi') \Big(\frac{T}{2\pi e}\Big)^{d'd[K:\Q]}\Big)+O(\log\q(iT,\pi\otimes\pi')).
\end{equation}
Thus it suffices to prove Theorems \ref{main-theorem} and \ref{main-theorem-2} for $1-\sigma$ sufficiently small.  Since the left side of Theorem \ref{main-theorem} is a decreasing function of $\sigma$ and the right side of Theorem \ref{main-theorem} is essentially constant for $1-\sigma\ll\mathcal{L}^{-1}$, it suffices to prove the theorem for $1-\sigma\gg\mathcal{L}^{-1}$.  Therefore, we may assume that $\Cl[abcon]{lowerbound}\leq \sigma\leq 1-\Cr{zero-free}\mathcal{L}^{-1}$, where $\frac{1}{2}<\Cr{lowerbound}<1$.  Since $\Cr{zero-free}$ and $\Cr{lowerbound}$ are absolute, we may take $\Cr{zero-free}$ sufficiently close to 0 and $\Cr{lowerbound}$ sufficiently close to $1$ such that we may take $\eta=\sqrt{2}(1-\sigma)$ in Proposition \ref{4.2}.

Suppose that $T\geq1$ and $\rho=\beta+i\gamma$ satisfies $|\gamma|\leq T$ and $\sigma\leq \beta\leq1-\Cr{zero-free}\mathcal{L}^{-1}$.  (In particular, $\rho$ is not a Siegel zero.)  For $\tau\in\R$, let
\[
\psi_{\rho}(\tau)=\begin{cases}
	1&\mbox{if $|\gamma-\tau|\leq 1-\sigma$ and $|\tau|\leq T$,}\\
	0&\mbox{otherwise.}
\end{cases}
\]
Since $\eta=\sqrt{2}(1-\sigma)$, we have that $\sqrt{2}\int_{-T}^T \psi_{\rho}(\tau)d\tau\geq\eta$ for each $\rho$, while $\sum_{\rho}\psi_{\rho}(\tau)\ll \eta\mathcal{L}$ by Lemma \ref{1.10}.  Thus
\[
N_{\pi\otimes\pi'}(\sigma,T)\ll\int_{-T}^{T}\sum_{\rho}\eta^{-1}\psi_{\rho}(\tau)d\tau.
\]
Since $\psi_{\rho}(\tau)\neq0$ implies that $|\rho-(1+i\tau)|\leq\eta$, we have by Proposition \ref{4.2} and the bound $1\ll\eta\mathcal{L}$ that
\begin{align*}
\int_{-T}^{T}\sum_{\rho}\eta^{-1}\psi_{\rho}(\tau)d\tau&\ll\int_{-T}^T \eta^{-1}\Big(\sum_{\rho}\psi_{\rho}(\tau)\Big)\frac{y^{\Cr{yexp2}\eta}}{(\log y)^{3}}\Big(\int_y^{y^{\Cr{yexp3}}}|S_{y,u}(\tau,\pi\otimes\pi')|^2\frac{du}{u}\Big)d\tau\\
&\ll \mathcal{L}\frac{y^{\Cr{yexp2} \eta}}{(\log y)^{3}}\int_y^{y^{\Cr{yexp3}}}\Big(\int_{-T}^T  |S_{y,u}(\tau,\pi\otimes\pi')|^2d\tau\Big)\frac{du}{u}.
\end{align*}

If $L(s,\pi,K)$ satisfies GRC, then it follows from Part 1 of Lemma \ref{main-inequality}, the definition of $S_{y,u}(\tau,\pi\otimes\pi')$, and the fact that $y=e^{\Cr{yexp1}\mathcal{L}}$ (with $\Cr{yexp1}$ sufficiently large) that
\[
N_{\pi\otimes\pi^\prime}(\sigma,T)\ll (d^\prime)^2\mathcal{L}\frac{y^{\Cr{yexp2} \eta}}{(\log y)^{4}}\int_y^{y^{\Cr{yexp3}}}\frac{(\log u)(\log u+d^2\log \q(\pi))}{u}du\ll d^2 y^{\Cr{yexp2} \eta}.
\]
Since $\eta=\sqrt{2}(1-\sigma)$ and $y=e^{\Cr{yexp1}\mathcal{L}}$ with $\Cr{yexp1}$ sufficiently large, we have
\[
N_{\pi\otimes\pi^\prime}(\sigma,T)\ll (d^\prime)^2y^{\Cr{yexp2} \sqrt{2}(1-\sigma)}\ll (d^\prime)^2(\mathcal{Q} T^{[K:\Q]})^{\Cr{yexp1}\Cr{yexp2}\sqrt{2}\mathcal{D}(1-\sigma)}.
\]
To conclude the proof of Theorem \ref{main-theorem}, we let $\Cr{c1}=2\sqrt{2}\Cr{yexp1}\Cr{yexp2}$.   Theorem \ref{main-theorem-2} is proved almost exactly the same way as Theorem \ref{main-theorem} except we use Part 2 of Lemma \ref{main-inequality} (see also the proof of \cite[Theorem 6]{Gallagher2}.) We omit the proof.

\subsection{Bounds on derivatives}
\label{sec:derivative_bounds}

We begin by introducing notation which we will use throughout this section and the next.  First, let $r=r(\pi\otimes\pi^\prime)$ be the order of the possible pole of $L_S(s,\pi\otimes\pi^\prime,K)$ at $s=1$.  We suppose that $L_S(s,\pi\otimes \pi^\prime,K)$ has a zero $\rho_0$ (which is not a Siegel zero) satisfying
\[
|\rho_0-(1+i\tau)|\leq \eta,
\]
and we set
\[
F(s)=\frac{L_S'}{L_S}(s,\pi\otimes\pi^\prime,K).
\]
Suppose that $|\tau|\leq T$, where $T\geq 1$, as in the statement of Proposition \ref{4.2}.  On the disk $|s-(1+i\tau)|<1/4$, by part 1 of Lemma \ref{1.10}, we have
\begin{equation*}
F(s)+\frac{r }{s}+\frac{r }{s-1}=\sum_{|\rho-(1+i\tau)|\leq1/2} \frac{1}{s-\rho}+G(s),
\end{equation*}
where $G(s)$ is analytic and $|G(s)|\ll \mathcal{L}$.  Setting $\xi=1+\eta+i\tau$, we have
\begin{equation}
\label{eqn:Fprime}
\frac{(-1)^ k  }{ k  !}\frac{d^ k   F}{ds^ k  }(\xi)+r (\xi-1)^{-( k  +1)}=\sum_{|\rho-(1+i\tau)|\leq1/2}(\xi-\rho)^{-( k  +1)}+O(8^{ k } \mathcal{L}),
\end{equation}
where the error term absorbs the contribution from integrating $G(s)$ over a circle of radius 1/8 centered at $\xi$ and the term coming from differentiating $\frac{r }{s}$.  We begin by obtaining a lower bound on the derivatives of $F(s)$.

\begin{lemma}
\label{lem:lower_bound_F}
Assume the notation above.  For any $M\gg \eta\mathcal{L}$, there is some $ k  \in [M,2M]$ such that
\[
\frac{\eta^{ k +1}}{ k !}\Big|\frac{d^ k  F}{ds^ k }(\xi)\Big|\geq \frac{1}{2}(100)^{-( k +1)},
\]
where $\xi = 1+\eta + i\tau$.
\end{lemma}

We prove Lemma \ref{lem:lower_bound_F} by using a version of Tur\'an's \cite{Turan} power-sum estimate.

\begin{lemma}[Tur\'an]
\label{4.1}
Let $z_1,\ldots,z_m\in\mathbb{C}$.  If $M\geq m$, then there exists $ k \in \Z\cap[M,2M]$ such that $|z_1^{ k }+\cdots+z_m^{ k }|\geq(\frac{1}{50}|z_1|)^{ k }.$
\end{lemma}

Let $M=300\eta\log y$.  By our choices of $\eta$, $\mathcal{L}$, $y$, $M$, and $k$, we have the useful relationship
\begin{equation}
\label{kappa-ineq}
1\ll \eta\mathcal{L}\asymp \eta\log y\asymp M\asymp k.
\end{equation}

\begin{proof}[Proof of Lemma \ref{lem:lower_bound_F}]
We begin by considering the contribution to \eqref{eqn:Fprime} from those zeros $\rho$ satisfying $200\eta < |\rho-(1+i\tau)| \leq 1/2$.  In particular, by decomposing the sum dyadically and applying part 2 of Lemma \ref{1.10}, we find that
\begin{equation*}
\sum_{200\eta<|\rho-(1+i\tau)|\leq1/2}|\rho-\xi|^{-( k +1)}\ll \sum_{j=0}^{\infty}(2^j 200\eta)^{-( k +1)}2^{j+1}r\mathcal{L}\ll (200\eta)^{- k }\mathcal{L},
\end{equation*}
This shows that it suffices to consider the zeros $\rho$ for which $|\rho-(1+i\tau)|\leq200\eta$.

Since $0<\eta\leq 1/55$, we have
\begin{equation}
\label{ineq1}
\frac{1}{ k !}\frac{d^ k   F}{ds^ k  }(\xi)+r(\xi-1)^{-( k +1)}\geq\Big|\sum_{|\rho-(1+i\tau)|\leq200\eta}(\xi-\rho)^{-( k +1)}\Big|-O((200\eta)^{- k }\mathcal{L}).
\end{equation} 
By Lemma \ref{1.10} (part 2), the sum over zeros has $\ll \eta\mathcal{L}$ terms.  Since $M\asymp \eta\mathcal{L}$, Lemma \ref{4.1} tells us that for some $ k \in[M,2M]$, the sum over zeros on the right side of \eqref{ineq1} is bounded below by $(50|\xi-\rho_0|)^{-( k +1)}$, where $\rho_0$ is the nontrivial zero which is being detected.

Since $|\xi-\rho_0|\leq 2\eta$, the right side of the above inequality is bounded below by
\begin{equation*}
(100\eta)^{-( k +1)}(1-O(2^{- k }\eta\mathcal{L})).
\end{equation*}
Since $ k \geq M\gg \eta\mathcal{L}$ and $\mathcal{L}^{-1}\ll\eta\ll 1$, there is a constant $0<\theta<1$ so that
\begin{equation*}
2^{- k }\eta\mathcal{L}=O(\theta^{\eta\mathcal{L}}\eta\mathcal{L})\leq1/4.
\end{equation*}
Therefore, for some $ k \in[M,2M]$ with $M\gg \eta\mathcal{L}$, we have
\begin{equation*}
\label{[a,b]}
\frac{\eta^{ k +1}}{ k !}\Big|\frac{d^ k  F}{ds^ k }(\xi)\Big|+r \eta^{ k +1}|(\xi-1)^{-( k +1)}|\geq\frac{3}{4}(100)^{-( k +1)}.
\end{equation*}
During the proof of Theorem 4.2 in \cite{weiss}, Weiss proves that
\begin{equation*}
\label{iii}
r  \eta^{ k +1}|(\xi-1)^{-( k +1)}|\leq \frac{1}{4}(100)^{-( k +1)}.
\end{equation*}
The desired result now follows, that $\frac{\eta^{ k +1}}{ k !}|\frac{d^ k  F}{ds^ k }(\xi)|\geq \frac{1}{2}(100)^{-( k +1)}$.
\end{proof}

We now turn to obtaining an upper bound on the derivatives of $F(s)$, for which we have the following.

\begin{lemma}
\label{lem:upper_bound_F}
Assume the notation preceeding Lemma \ref{lem:lower_bound_F}.  Let $M=300\eta \log y$, and let $k$ be determined by Lemma \ref{lem:lower_bound_F}.  Then
\[
\frac{\eta^{ k +1}}{ k !}\Big|\frac{d^ k  F}{ds^ k }(\xi)\Big|\leq \eta^2\int_y^{y^{\Cr{yexp3}}}|S_{y,u}(\tau,\pi\otimes\pi')|\frac{du}{u}+\frac{1}{4}(100)^{-( k +1)},
\]
where $S_{y,u}(\tau,\pi\otimes\pi^\prime)$ is as in Proposition \ref{4.2}.
\end{lemma}

\begin{proof}
Let $M=300\eta\log y$ and $y=e^{\Cr{yexp1}\mathcal{L}}$ for some $\Cr{yexp1}$, which we will take to be sufficiently large.  For $u>0$, define 
\[
j_ k (u)=\frac{u^ k  e^{-u}}{ k !},
\]
which satisfies
\begin{equation*}
j_ k (u)\leq\begin{cases}
(100)^{- k }&\mbox{if $u\leq  k /300$,}\\
(110)^{- k }e^{-u/2}&\mbox{if $u\geq 20 k $.}
\end{cases}
\end{equation*}
Letting $\Cr{yexp3}\geq12000$ be sufficiently large, we thus have
\begin{equation}
\label{eqn:j-ineq-2}
j_ k (\eta\log(\mathrm{N}\a))\leq\begin{cases}
(110)^{- k }&\mbox{if $\mathrm{N}\a\leq y$,}\\
(100)^{- k }(\mathrm{N}\a)^{-\eta/2}&\mbox{if $\mathrm{N}\a\geq y^{\Cr{yexp3}}$}.
\end{cases}
\end{equation}
Differentiating the Dirichlet series for $F(s)$ directly, we obtain
\begin{equation*}
\label{RP-sum}
\frac{(-1)^{ k +1}\eta^{ k +1}}{ k !}\frac{d^ k   F}{ds^ k  }(\xi)=\eta\sum_{(\a, S )=1}\frac{\Lambda_{\pi\otimes\pi^\prime}(\a)}{\mathrm{N}\a^{1+i\tau}}j_ k (\eta\log(\mathrm{N}\a))
\end{equation*}
Splitting the above sum $\sum$ in concert with the inequality \eqref{eqn:j-ineq-2} and suppressing the summands, we write
\[
\sum=\sum_{\mathrm{N}\p\in(0,y]\cup(y^{\Cr{yexp3}},\infty)}+\sum_{\textup{$\a$ not prime}}+\sum_{y<\mathrm{N}\p\leq y^{\Cr{yexp3}}}.
\]
We will estimate these three sums separately.

We use Lemma \ref{1.11} and \eqref{kappa-ineq} to obtain
\begin{align*}
\Big|\eta\sum_{\mathrm{N}\p\in(0,y]\cup(y^{\Cr{yexp3}},\infty)}\Big|&\ll \eta(110)^{- k }\Big(\sum_{\substack{\mathrm{N}\a\leq y \\ (\a, S )=1}}\frac{|\Lambda_{\pi\otimes\pi'}(\a)|}{\mathrm{N}\a}+\sum_{(\a, S )=1}\frac{|\Lambda_{\pi\otimes\pi'}(\a)|}{\mathrm{N}\a^{1+\eta/2}}\Big)\notag\\
&\ll \eta(110)^{- k }\Big(\frac{1}{\eta}+\log y+d'd\log(\q(\pi)\q(\pi'))\Big)\notag\\
&\ll (110)^{- k }(1+\eta\log y+\eta\mathcal{L})\ll  k (110)^{- k }.
\end{align*}
Since $\eta\leq 1/55$, the identity $\sum_{m\geq0}j_m(u)=1$ implies that
\[
\mathrm{N}\a^{-1/2}j_ k (\eta\log(\mathrm{N}\a))=(2\eta)^ k  \mathrm{N}\a^{-\eta}j_ k (\log(\mathrm{N}\a)/2)\leq (110)^{- k }\mathrm{N}\a^{-\eta}.
\]
Thus, as above,
\begin{equation*}
\label{ineq2}
\Big| \eta\sum_{\textup{$\a$ not prime}}\Big|\ll \eta(110)^{- k }\sum_{\substack{\a=\p^m \\ m\geq 2 \\ (\a, S )=1}}\frac{|\Lambda_{\pi\otimes\pi'}(\a)|}{\mathrm{N}\a^{1/2+\eta}}\ll \eta(110)^{- k }\sum_{(\a, S )=1}\frac{|\Lambda_{\pi\otimes\pi'}(\a)|}{\mathrm{N}\a^{1+2\eta}}\ll  k (110)^{- k }.
\end{equation*}
as well.  Finally, recall that
\begin{equation*}
\label{dirichlet-poly}
S_{y,u}(\tau,\pi\otimes\pi')=\sum_{y<\mathrm{N}\p\leq u}\frac{\Lambda_{\pi\otimes\pi^\prime}(\p)}{\mathrm{N}\p^{1+i\tau}}.
\end{equation*}
Summation by parts gives us
\begin{equation*}
\sum_{y<\mathrm{N}\p\leq y^{\Cr{yexp3}}}=S_{y,y^{\Cr{yexp3}}}(\tau,\pi\otimes\pi')j_ k (\eta\log y^{\Cr{yexp3}})-\eta\int_y^{y^{\Cr{yexp3}}}S_{y,u}(\tau,\pi\otimes\pi')j_ k '(\eta\log u)\frac{du}{u}
\end{equation*}
since $S_{y,y}(\tau,\pi\otimes\pi')=0$.  Much like the above calculations,
\begin{equation*}
\label{ineq3}
|\eta S_{y,y^{\Cr{yexp3}}}(\tau,\pi\otimes\pi')j_ k (\eta\log y^{\Cr{yexp3}})|\ll \eta(110)^{- k }y^{-\Cr{yexp3}\eta/2}\sum_{\substack{\mathrm{N}\p\leq y^{\Cr{yexp3}} \\ (\p, S )=1}}\frac{|\Lambda_{\pi\otimes\pi'}(\p)|}{\mathrm{N}\p}\ll k (110)^{- k }.
\end{equation*}
Therefore, since $|j_ k '(u)|=|j_{ k -1}(u)-j_ k (u)|\leq j_{ k -1}(u)+j_ k (u)\leq 1$, we have
\begin{equation*}
\Big|\eta\sum_{y<\mathrm{N}\p\leq y^{\Cr{yexp3}}}\Big|\leq \eta^2\int_y^{y^{\Cr{yexp3}}}|S_{y,u}(\tau,\pi\otimes\pi')|\frac{du}{u}+O( k (110)^{- k }).
\end{equation*}
However, by \eqref{kappa-ineq} and the bound $\eta\gg\mathcal{L}^{-1}$, it follows that if $ k $ is sufficiently large, then each term of size $O( k (110)^{- k })$ is at most $\frac{1}{16}(100)^{-( k +1)}$.  The lemma follows.
\end{proof}

\subsection{Zero detection: The proof of Proposition \ref{4.2}}
\label{sec:zero_detection}

We now combine our upper and lower bounds on the derivatives of $F$ to prove Proposition \ref{4.2}.  Thus, we wish to show that if $\rho_0$ is a zero satisfying $|\rho_0-(1+i\tau)|\leq \eta$, then
\[
\frac{y^{\Cr{yexp2} \eta}}{(\log y)^3}\int_y^{y^{\Cr{yexp3}}}|S_{y,u}(\tau,\pi\otimes\pi')|^2\frac{du}{u}\gg1.
\]

Combining Lemmas \ref{lem:lower_bound_F} and \ref{lem:upper_bound_F}, we find that
\[
\eta^2\int_y^{y^{\Cr{yexp3}}}|S_{y,u}(\tau,\pi\otimes\pi')|\frac{du}{u}\geq\frac{1}{4}(100)^{-( k +1)}.
\]
Using \eqref{kappa-ineq}, we have
\[
\eta^2\int_y^{y^{\Cr{yexp3}}}|S_{y,u}(\tau,\pi\otimes\pi')|\frac{du}{u}\gg y^{-\Cr{yexp2}\eta/4},
\]
where $\Cr{yexp2}$ is sufficiently large.  Multiplying both sides by $y^{-\Cr{yexp2}\eta/4}$ yields
\[
y^{-\Cr{yexp2}\eta/4}\eta^2\int_y^{y^{\Cr{yexp3}}}|S_{y,u}(\tau,\pi\otimes\pi')|\frac{du}{u}\gg y^{-\Cr{yexp2}\eta/2}.
\]
Using \eqref{kappa-ineq} again, we have that $y^{-\Cr{yexp2}\eta/4}\eta^2\ll(\log y)^{-2}$, so
\[
\frac{1}{(\log y)^2}\int_y^{y^{\Cr{yexp3}}}|S_{y,u}(\tau,\pi\otimes\pi')|\frac{du}{u}\gg y^{-\Cr{yexp2}\eta/2}.
\]
Squaring both sides and applying the Cauchy-Schwarz inequality yields the proposition.

\subsection{Proof of Theorem \ref{thm:no_GRC}}

We conclude this section with the proof of Theorem \ref{thm:no_GRC}.  We content ourselves with proving the first statement, as the second statement follows along exactly the lines using Theorem \ref{main-theorem-2}.  Recall that in this setup, $\pi \in \mathcal{A}_d(K)$ and $\pi^\prime \in \mathcal{A}_{d^\prime}(K)$, where each of $d$ and $d^\prime$ is at most 2.  If either $\pi$ or $\pi^\prime$ satisfies GRC, then Theorem \ref{main-theorem} applies, and we obtain Theorem \ref{thm:no_GRC} as an immediate consequence.  If neither $\pi$ nor $\pi^\prime$ satisfies GRC, then necessarily $d=d^\prime =2$, and we may split further into two cases: either there exists an id\`ele class character $\chi$ such that $\pi \cong \pi^\prime \otimes \chi$ or not.  If there is such a character $\chi$, then we have $L(s,\pi \otimes   \pi^\prime, K) = L(s,\omega_\pi\otimes \chi,K) L(s,\mathrm{Sym}^2\pi \otimes \chi,K)$, where $\omega_\pi$ is the central character of $\pi$.  Since $\pi$ is not monomial (if it were, it would satisfy GRC), $\mathrm{Sym}^2\pi \in \mathcal{A}_3(K)$ by Gelbart-Jacquet \cite{GJ}, so we may apply Theorem \ref{main-theorem} to the two $L$-functions $L(s,\omega_\pi\otimes\chi,K)$ and $L(s,\mathrm{Sym}^2\pi\otimes \chi,K)$, which yields Theorem \ref{thm:no_GRC} in this case.  Finally, if there is no character $\chi$ such that $\pi^\prime \cong \pi \otimes \chi$, then by work of Ramakrishnan \cite[Theorem M]{MR1792292}, we have $\pi \otimes \pi^\prime \in \mathcal{A}_4(K)$ since again neither $\pi$ nor $\pi^\prime$ is monomial, lest it would satisfy GRC.  Thus, in this case, we may appeal to Corollary \ref{cor:lfzde} directly, and Theorem \ref{thm:no_GRC} follows in all cases.

\section{Proof of Theorems \ref{effective-hoheisel}, \ref{thm:hoheisel_sato_tate}, and \ref{thm:least_sato_tate}}

\label{sec:applications}

In this section, we consider the arithmetic applications of the zero-density estimates provided in Theorem \ref{main-theorem} and Corollary \ref{cor:lfzde} to approximate versions of Hoheisel's short interval prime number theorem.  We prove Theorems \ref{effective-hoheisel} and \ref{thm:hoheisel_sato_tate}, and Theorem \ref{thm:least_sato_tate} follows readily from Theorem \ref{thm:hoheisel_sato_tate}.

\subsection{Proof of Theorem \ref{effective-hoheisel}}
\label{subsec:general}
We first prove the explicit formula \eqref{eqn:shakedown_street_2} for the right hand side of Theorem \ref{effective-hoheisel} without making reference to the size of $\Lambda_{\pi\otimes\tilde{\pi}}(\a)$.  Note that the analogous result in \cite[Proof of Theorem 1.4]{AT} requires that the mean value of $\Lambda_{\pi\otimes\tilde{\pi}}(\a)$ remain bounded over short intervals, and the analogous result in \cite[Proof of Theorem 1.1]{Motohashi} requires an asymptotic estimate for a certain sum of Dirichlet coefficients with a power-saving error term.  Our explicit formula only uses the well-known fact that $\Lambda_{\pi\otimes\tilde{\pi}}(\a)\geq0$ for all $\a$; it holds regardless of whether $\pi$ is self-dual.

Let $x\geq2$.  Define
\[
\psi_{\pi\otimes\tilde{\pi}}(x)=\int_0^x \Big(\sum_{\mathrm{N}\a\leq t}\Lambda_{\pi\otimes\tilde{\pi}}(\a)\Big)dt=-\frac{1}{2\pi i}\int_{2-i\infty}^{2+i\infty}\frac{L'}{L}(s,\pi\otimes\tilde{\pi},K)\frac{x^{s+1}}{s(s+1)}ds.
\]
Note that the sum in the first integrand is monotonically increasing.  Thus if $0<y<x$ and $x+y>1$, then by the mean value theorem,
\begin{equation}
\label{eqn:mvt}
-\frac{\psi_{\pi\otimes\tilde{\pi}}(x-y)-\psi_{\pi\otimes\tilde{\pi}}(x)}{y}\leq \sum_{\mathrm{N}\a\leq x}\Lambda_{\pi\otimes\tilde{\pi}}(\a)\leq \frac{\psi_{\pi\otimes\tilde{\pi}}(x+y)-\psi_{\pi\otimes\tilde{\pi}}(x)}{y}.
\end{equation}
By a standard residue theorem computation,
\[
\psi_{\pi\otimes\tilde{\pi}}(x)=\frac{x^2}{2}-\sum_{\rho\neq0,-1}\frac{x^{\rho+1}}{\rho(\rho+1)}-(\mathrm{Res}_{s=0}+\mathrm{Res}_{s=-1})\frac{L'}{L}(s,\pi\otimes\tilde{\pi},K)\frac{x^{s+1}}{s(s+1)},
\]
where $\rho$ runs over all zeros of $L(s,\pi\otimes\tilde{\pi},K)$.  Thus
\begin{align}
\label{eqn:captain}
\pm\frac{\psi_{\pi\otimes\tilde{\pi}}(x\pm y)-\psi_{\pi\otimes\tilde{\pi}}(x)}{y}&=x\pm\frac{y}{2}\mp\sum_{\rho\neq0,-1}\frac{x^{\rho+1}((1\pm \frac{y}{x})^{\rho+1}-1)}{y\rho(\rho+1)}\notag\\
&\mp (\mathrm{Res}_{s=0}+\mathrm{Res}_{s=-1})\frac{L'}{L}(s,\pi\otimes\tilde{\pi},K)\frac{x^{s+1}((1\pm \frac{y}{x})^{s+1}-1)}{ys(s+1)}.
\end{align}

We first address the sum over zeros, restricting our attention to those $\rho=\beta+i\gamma$ for which $0<\beta<1$.  Observe that for each such $\rho$,
\[
\mp\frac{x^{\rho+1}((1\pm\frac{y}{x})^{\rho+1}-1)}{y\rho(\rho+1)}=-\frac{x^\rho}{\rho}\mp y w_\rho^{\pm} x^{\rho-1},
\]
where
\[
w_\rho^{\pm}:=\frac{(1\pm\frac{y}{x})^{\rho+1}-1\mp(\rho+1)\frac{y}{x}}{\rho(\rho+1)(\frac{y}{x})^2}.
\]
Since $0<\beta<1$, $0<y<x$ and $x+y>1$, a minor change in the proof of \cite[Lemma 2.1]{MR3434887} yields the bound $|w_{\rho}^{\pm}|\leq 1$ in both $\pm$ cases.  Thus for any $1\leq T\leq x$, the sum over zeros $\rho=\beta+i\gamma$ with $0<\beta<1$ in \eqref{eqn:captain} equals
\begin{equation}
\label{eqn:big_eyed_fish}
-\sum_{\substack{|\gamma|\leq T \\ 0<\beta<1}}\frac{x^\rho}{\rho}+O\Big(\sum_{\substack{|\gamma|\geq T \\ 0<\beta<1}}\Big|\frac{x^{\rho+1}((1\pm\frac{y}{x})^{\rho+1}-1)}{y\rho(\rho+1)}\Big|+y\sum_{\substack{|\gamma|\leq T \\ 0<\beta<1}}x^{\beta-1}\Big).
\end{equation}
Using \eqref{eqn:zero-count} and the fact that $0<y<x$, we see that the first sum over zeros in the error term of \eqref{eqn:big_eyed_fish} is
\[
\ll \frac{x^2}{y}\sum_{\substack{|\gamma|\geq T \\ 0<\beta<1}}\frac{1}{|\rho|^2}\ll [K:\Q]d^2\frac{x^2 (\log T)\log \mathfrak{q}(\pi)}{y T}.
\]
We choose $T\geq1$, $A=4\Cr{c1}d^2$, $x=(\q(\pi)^2 T^{[K:\Q]})^A$, and 
\[
y=x^{1-\frac{1}{2A}}(\log x)\Big(\sum_{\substack{|\gamma|\leq T \\ 0<\beta<1}}x^{\beta-1}\Big)^{-1/2}
\]
so that the sum over zeros $\rho=\beta+i\gamma$ in \eqref{eqn:captain} equals
\begin{equation}
\label{eqn:zero_error}
-\sum_{\substack{|\gamma|\leq T \\ 0<\beta<1}}\frac{x^{\rho}}{\rho}+O\Big(x^{1-\frac{1}{2A}}(\log x)\Big(\sum_{|\gamma|\leq T}x^{\beta-1}\Big)^{1/2}\Big).
\end{equation}

With our choice of $y$, the contribution to the sum over zeros $\rho=\beta+i\gamma$ in \eqref{eqn:captain} with $\beta\leq0$ is smaller than the error term in \eqref{eqn:zero_error}.  The same can be said for the contribution from the residues in \eqref{eqn:captain}, which can be computed using \cite[Equation 5.24]{IK}.  Collecting all of our estimates, we now see that
\begin{equation}
\label{eqn:shakedown_street_2}
\sum_{\mathrm{N}\a\leq x}\Lambda_{\pi\otimes\tilde{\pi}}(\a)= x-\sum_{\substack{|\gamma|\leq T \\ 0<\beta<1}}\frac{x^{\rho}}{\rho}+O\Big(x^{1-\frac{1}{2A}}(\log x)\Big(\sum_{|\gamma|\leq T}x^{\beta-1}\Big)^{1/2}\Big).
\end{equation}
We now see that for any $1\leq h\leq x$,
\begin{align}
\label{eqn:shakedown_street}
\Big|\sum_{x<\mathrm{N}\a\leq x+h}\Lambda_{\pi\otimes\tilde{\pi}}(\a)-h\Big|&= \Big|-\sum_{\substack{|\gamma|\leq T \\ 0<\beta<1}}\frac{(x+h)^{\rho}-x^\rho}{\rho}+O\Big(x^{1-\frac{1}{2A}}(\log x)\Big(\sum_{\substack{|\gamma|\leq T \\ 0<\beta<1}}x^{\beta-1}\Big)^{1/2}\Big)\Big|\notag\\
&\leq h\sum_{\substack{|\gamma|\leq T \\ 0<\beta<1}}x^{\beta-1}+O\Big(x^{1-\frac{1}{2A}}(\log x)\Big(\sum_{\substack{|\gamma|\leq T \\ 0<\beta<1}}x^{\beta-1}\Big)^{1/2}\Big).
\end{align}

To bound the sums over zeros in \eqref{eqn:shakedown_street}, note that by the functional equation for $L(s,\pi\otimes\tilde{\pi},K)$ and the zero-free region in Lemma \ref{1.9} that
\begin{align*}
	\sum_{|\gamma|\leq T}x^{\beta-1}\leq x^{\beta_{\mathrm{exc}}-1}+2\int_{1/2}^{1-\Cr{zero-free}/\mathcal{L}} x^{\sigma-1}dN_{\pi\otimes\tilde{\pi}}(\sigma,T).
\end{align*}
For simplicity of calculations, we observe from Theorem \ref{main-theorem} and our choice of $x$ and $T$ that
\begin{equation}
\label{eqn:L_to_x}
\mathcal{L}=\frac{1}{4\Cr{c1}}\log x,\qquad N_{\pi\otimes\tilde{\pi}}(\sigma,T)\ll d^2 (\q(\pi)^2T^{[K:\Q]})^{\Cr{c1}d^2(1-\sigma)}= d^2 x^{\frac{1}{4}(1-\sigma)}.
\end{equation}
If $\pi$ satisfies GRC, then we use Theorem \ref{main-theorem}, Lemma \ref{1.9}, and \eqref{eqn:L_to_x} to obtain
\begin{align}
\label{eqn:bound_sum_zeros_LFZDE}
\int_{1/2}^{1-\Cr{zero-free}/\mathcal{L}} x^{\sigma-1}\, dN_{\pi\otimes\tilde{\pi}}(\sigma,T) &= x^{-1/2} N_{\pi\otimes\tilde{\pi}}(1/2,T) + \log x \int_{1/2}^{1-\Cr{zero-free}/\mathcal{L}} x^{\sigma-1} N_{\pi\otimes\tilde{\pi}}(\sigma,T)\,d\sigma\notag\\
	&\ll x^{-3/8} + d^2\log x \int_{1/2}^{1-4\Cr{c1}\Cr{zero-free}/\log x} x^{\frac{3}{4}(\sigma-1)} \,d\sigma\notag\\
	&\ll x^{-3/8}+d^2(x^{-3/8} + e^{-3\Cr{c1}\Cr{zero-free}}).
\end{align}
If $\pi\in\mathcal{A}_2(K)$ and GRC is not satisfied, then we apply Theorem \ref{thm:no_GRC} instead of Theorem \ref{main-theorem} and arrive at the same conclusion as \eqref{eqn:bound_sum_zeros_LFZDE}.

Applying \eqref{eqn:bound_sum_zeros_LFZDE} to bound the sum over zeros in \eqref{eqn:shakedown_street}, we find that
\[
\Big|\sum_{x<\mathrm{N}\a\leq x+h} \Lambda_{\pi\otimes\tilde{\pi}}(\a)-h\Big|\leq\Big(\Cl[abcon]{int-ineq2} d^2 e^{-3\Cr{c1}\Cr{zero-free}}+o_{h\to\infty}(1)\Big)h+O(x^{1-\frac{1}{2A}}\log x)
\]
where the $o_{h\to\infty}(1)$ term is the contribution from the possible Siegel zeros and $\Cr{int-ineq2}$ is sufficiently large.  
Because $\Cr{c1}$ is both large and absolute, we may replace $\Cr{c1}$ with the larger constant $\max\{\Cr{c1},(3\Cr{zero-free})^{-1}\log(4 \Cr{int-ineq2} d^2)\}$.  This yields
\[
\Big|\sum_{x<\mathrm{N}\a\leq x+h} \Lambda_{\pi\otimes\tilde{\pi}}(\a)
	-h\Big|\leq \Big(\frac{1}{4}+o_{h\to\infty}(1)\Big)h+ O(x^{1-\frac{1}{2A}}\log x).
\]
Finally, taking $h\gg  x^{1-\frac{1}{2A}}(\log x)^2$, the theorem follows when $T$ (hence $x$) is sufficiently large.

\subsection{The Sato-Tate conjecture}
\label{subsec:sato-tate}

Following Shahidi \cite[pg. 162]{Shahidi}, we now specify the representations $\pi$ for which we expect the Sato-Tate conjecture to hold.  Let $K$ be a totally real field, and let $\pi\in\mathcal{A}_2(K)$ be non-CM and have trivial central character.  We call $\pi$ {\bf genuine} when $\pi$ is not a twist by an id{\'e}le class character of either a monomial representation or a representation of Galois type.  (A monomial representation is a representation $\rho$ for which $\rho\otimes\chi\cong\rho$ for some nontrivial character $\chi$ of $K^\times \setminus \mathbb{A}_K^{\times}$.  A representation $\rho$ is of Galois type if $\rho_{\v}$ factors through the Galois group of $\bar{K}_{\v}/K_{\v}$ for every archimedean place $\v$ of $K$.)

Examples of genuine $\pi$ include those associated to (1) newforms of even weight $k\geq2$ and trivial character, (2) modular elliptic curves, (3) Hilbert modular forms, and (4) Hecke-Maass forms.  The hypothesis of GRC is known to hold for most of these examples: for newforms by Deligne \cite{Deligne}, for elliptic curves by Hasse (see Silverman \cite[Chapter 5]{Silverman}), and for Hilbert modular forms over totally real number fields with each weight both even and at least 2 by Blasius \cite{Blasius}.

Recall that the Sato-Tate conjecture concerns the distribution of the quantities $\lambda_\pi(\p)=2\cos\theta_\p$ as $\p$ ranges over primes for which $\pi_\p$ is unramified, where $\theta_\p\in[0,\pi]$.  At each such prime $\p$, the local factor of the $n$-th symmetric power $L$-function is given by
\[
L_\p(s,\mathrm{Sym}^n\pi,K) =\prod_{j=0}^n (1-e^{i\theta_\p (n-2j)} \mathrm{N}\p^{-s})^{-1} = \sum_{k=0}^\infty \frac{U_n(\cos(k\theta_{\p}))}{\mathrm{N}\p^{s}}, 
\]
where $U_n$ is the $n$-th Chebyshev polynomial of the second kind.  At ramified primes $\p$, it follows from \cite{MS} there are numbers $\beta_{\mathrm{Sym}^n\pi}(j,\p)$ of absolute value at most $ \mathrm{N}\p^{\frac{1}{2}-\frac{1}{(n+1)^2+1}}$ for which the local factor is given by
\[
L_\p(s,\mathrm{Sym}^n\pi,K)=\prod_{j=0}^n (1-\beta_{\mathrm{Sym}^n\pi}(j,\p)\mathrm{N}\p^{-s})^{-1}.
\]
(If $\p$ is ramified, then some of the $\beta_{\mathrm{Sym}^n\pi}(j,\p)$ might equal zero.)  We note that $L(s,\mathrm{Sym}^1\pi,K)=L(s,\pi,K)$ and $L(\mathrm{Sym}^0 \pi,K)=\zeta_K(s)$.

In Theorem \ref{thm:hoheisel_sato_tate}, our goal is to estimate for $I\subseteq[-1,1]$ the summation
\begin{equation}
\label{eqn:satotategoal}
\sum_{\substack{x<\mathrm{N}\p\leq x+h \\ (\p,S)=1}}\mathbf{1}_{I}(\cos\theta_{\p})\log \mathrm{N}\p
\end{equation}
where $S$ is the set of $\p$ for which $\pi$ is ramified and $h\geq x^{1-\delta}$ for some $\delta>0$.  Recall from the discussion before Theorem \ref{thm:hoheisel_sato_tate} in Section \ref{sec:Section1} that $I$ can be $\mathrm{Sym}^N$-minorized if there exist $b_0,\dots,b_N\in\R$ such that $b_0>0$ and \eqref{eqn:symn-minor1} holds for all $t\in[-1,1]$.  Thus, if $I$ can be $\mathrm{Sym}^n$-minorized, we can obtain a non-trivial lower bound for \eqref{eqn:satotategoal} by considering an appropriate linear combination of the logarithmic derivatives of $L(s,\mathrm{Sym}^n \pi,K)$ for $n\leq N$.

\begin{proof}[Proof of Theorem \ref{thm:hoheisel_sato_tate}]

The upper bound follows from GRC at the unramified primes and the Brun-Titchmarsh theorem \cite[Corollary 2]{MR0374060}, so we proceed to the lower bound.  Suppose that $I\subset[-1,1]$ can be $\mathrm{Sym}^n$-minorized and that $L(s,\mathrm{Sym}^n \pi,K)$ is automorphic for each $0\leq n\leq N$.  Let $b_0,\dots,b_n$ be as in \eqref{eqn:symn-minor1} and set $B=\max_{0\leq n\leq N}|b_n|/b_0$.  Let $T\geq1$, $A=4\Cr{c1}(N+1)^4$, and
\[
x=\max\Big\{D_K[K:\Q]^{[K:\Q]},T^{[K:\Q]}\max_{0\leq n\leq N}\q(\mathrm{Sym}^n\pi)\Big\}^A.
\]

First, observe that
\begin{equation}
\label{eqn:leftovers}
\sum_{\substack{\mathrm{N}\p\leq x \\ (\p,S)=1}}U_n(\cos\theta_{\p})\log\mathrm{N}\p=\sum_{\substack{\mathrm{N}\a\leq x \\ (\a,S)=1}}\Lambda_{\mathrm{Sym}^n\pi}(\a)-\sum_{\substack{\mathrm{N}\p^m\leq x \\ m\geq2 \\ (\p,S)=1}}\Lambda_{\mathrm{Sym}^n\pi}(\p^m).
\end{equation}
There are $\ll (\log x)\log\q(\mathrm{Sym}^n\pi)$ ramified prime powers whose norm is in the interval $[x,x+h]$.  Using \eqref{eqn:almost-GRC}, the contribution to the sum over prime powers on the right hand side of \eqref{eqn:leftovers} arising from the ramified prime powers is
\begin{equation}
\label{eqn:error_1_ST}
\ll n\sum_{\substack{\mathrm{N}\p^m\leq x \\ m\geq2 \\ (\p,S)=1}}\Lambda_{K}(\p^m)\mathrm{N}\p^{m/2}\ll n x^{1/2} (\log x)^3.
\end{equation}
Using GRC at the unramified primes, the contribution to the sum over prime powers on the right hand side of \eqref{eqn:leftovers} arising from the unramified prime powers is
\begin{equation}
\label{eqn:error_2_ST}
\ll n\sum_{\substack{\mathrm{N}\a\leq x \\ (\a,S)=1}}\Lambda_K(\a) \ll n[K:\Q]\sqrt{x}\ll n x^{1/2} (\log x)^3.
\end{equation}
This establishes the lower bound
\begin{equation}
\label{eqn:double_sum_ST}
\sum_{\substack{\mathrm{N}\p\leq x \\ (\p,S)=1}}\mathbf{1}_{I}(\cos\theta_{\p})\log \mathrm{N}\p\geq \sum_{n=0}^N b_n \sum_{\substack{\mathrm{N}\a\leq x \\  (\a,S)=1}}\Lambda_{\mathrm{Sym}^n\pi}(\a)-O(b_0 B N^2 x^{1/2}(\log x)^3).
\end{equation}

It is unclear whether the double sum in \eqref{eqn:double_sum_ST} is monotonically increasing for all $x$, so we cannot use the arguments from the proof of Theorem \ref{effective-hoheisel}.  Thus we begin with the standard Perron integral \cite[Chapter 17]{Davenport}; if $\sigma_0=1+\frac{1}{\log x}$, then
\[
\sum_{\substack{\mathrm{N}\a\leq x \\  (\a,S)=1}}\Lambda_{\mathrm{Sym}^n\pi}(\a)=-\frac{1}{2\pi i}\int_{\sigma_0-iT}^{\sigma_0+iT}\frac{L_S'}{L_S}(s,\mathrm{Sym}^n \pi)\frac{x^s}{s}ds+O\Big(x\sum_{(\a,S)=1}\frac{|\Lambda_{\mathrm{Sym}^n \pi}(\a)|}{\mathrm{N}\a^{\sigma_0}}\min\Big\{1,\frac{1}{T|\log\frac{x}{\mathrm{N}\a}|}\Big\}\Big),
\]
where $L_S(s,\mathrm{Sym}^n\pi,K)$ is the partial $L$-function with the Euler factors at ramified $\p$ removed.  Let $H\geq 2$.  If $|\mathrm{N}\a-x|> \frac{x}{H}$, then $|\log \frac{x}{\mathrm{N}\a}|\gg\frac{1}{H}$.  Therefore,
\begin{align}
\label{eqn:sum_with_ramified}
x\sum_{\substack{|\mathrm{N}\a-x| > \frac{x}{H} \\ (\a,S)=1}} \frac{|\Lambda_{\mathrm{Sym}^n \pi}(\a)|}{\mathrm{N}\a^{\sigma_0}}\min\Big\{1,\frac{1}{T|\log\frac{x}{\mathrm{N}\a}|}\Big\}\leq x\frac{H}{T}\sum_{\substack{(\a,S)=1}}\frac{|\Lambda_{\mathrm{Sym}^n \pi}(\a)|}{\mathrm{N}\a^{\sigma_0}}.
\end{align}
Using Lemma \ref{1.10} to bound the sum over unramified primes, we bound \eqref{eqn:sum_with_ramified} by $O(\frac{Hx}{T}(\log x)^2)$.

It remains to estimate
\[
x\sum_{\substack{|\mathrm{N}\a-x|\leq \frac{x}{H} \\ (\a,S)=1}}\frac{|\Lambda_{\mathrm{Sym}^n \pi}(\a)|}{\mathrm{N}\a^{\sigma_0}}
=x\sum_{\substack{|\mathrm{N}\p^m-x|\leq \frac{x}{H} \\ (\p,S)=1}} \frac{|\Lambda_{\mathrm{Sym}^n \pi}(\p^m)|}{(\mathrm{N}\p^m)^{\sigma_0}}.
\]
There are $\ll [K:\Q] x/H$ powers of unramified primes with norm between $x(1-\frac{1}{H})$ and $x(1+\frac{1}{H})$, so GRC at the unramified primes implies that the above display is bounded by $[K:\Q](n+1)x(\log x)/H\ll x(\log x)^2/H$.  Collecting the above estimates and recalling the definitions of $x$ and $T$, we choose $H=T^{1/2}$ to obtain
\[
\sum_{\substack{\mathrm{N}\a\leq x \\ (\a,S)=1}}\Lambda_{\mathrm{Sym}^n \pi}(\a)=-\frac{1}{2\pi i}\int_{\sigma_0-iT}^{\sigma_0+iT}\frac{L_S'}{L_S}(s,\mathrm{Sym}^n \pi)\frac{x^s}{s}ds+O(x^{1-\frac{1}{2A}}(\log x)^2).
\]

We deduce from a standard residue theorem computation that
\begin{equation}
\label{eqn:ST_explicit_formula}
\sum_{\substack{\mathrm{N}\a\leq x \\ (\a,S)=1}}\Lambda_{\mathrm{Sym}^n \pi}(\a)=r(\mathrm{Sym}^n\pi)x-\sum_{\substack{\rho=\beta+i\gamma \\ 0\leq\beta<1 \\ |\gamma|\leq T}}\frac{x^\rho}{\rho}+O\Big(\sum_{\p\in S}\sum_{j=1}^{n+1}\frac{x^{\rho_{\p}}}{\rho_{\p}}+x^{1-\frac{1}{2A}}(\log x)^2\Big),
\end{equation}
where $\rho_{\p}=(\log \beta_{\mathrm{Sym}^n\pi}(j,\p))/\log\mathrm{N}\p$.  Using \eqref{eqn:LRS} (which also holds at the ramified primes by the work in \cite{MS}) and the fact that $(n+1)\log \q(\pi)\ll\log x$, we find that the error term in \eqref{eqn:ST_explicit_formula} is $O(x^{1-\frac{1}{2A}}(\log x)^2)$. Thus the lower bound
\[
\sum_{x<\mathrm{N}\p\leq x+h} {\bf 1}_{I}(\cos\theta_{\p})\log\mathrm{N}\p\geq b_0 \Big(h-B\sum_{n=0}^N ~\sum_{\substack{\rho=\beta+i\gamma \\ |\gamma|\leq T \\ 0\leq \beta<1}}\frac{(x+h)^\rho-x^\rho}{\rho}-O(BN x^{1-\frac{1}{2A}}(\log x)^3)\Big)
\]
easily follows from \eqref{eqn:double_sum_ST} and \eqref{eqn:ST_explicit_formula}, where $\rho$ runs through the zeros of $L(s,\mathrm{Sym}^n\pi,K)$.

By a calculation nearly identical to \eqref{eqn:bound_sum_zeros_LFZDE}, we deduce the existence of a sufficiently large $\Cl[abcon]{int-ineq3}>0$ such that
\begin{align*}
&\sum_{x<\mathrm{N}\p\leq x+h} {\bf 1}_{I}(\cos\theta_{\p})\log\mathrm{N}\p \\
&\geq b_0\Big((1 - \Cr{int-ineq3}BN e^{-3\Cr{c1}\Cr{zero-free}}-o_{h\to\infty}(1))h-O(BN  x^{1-\frac{1}{2A}}(\log x)^3)\Big).
\end{align*}
The $o_{h\to\infty}(1)$ term arises from the contributions of the possible Siegel zeros, and we make $\Cr{c1}$ sufficiently large to account for the errors in \eqref{eqn:error_1_ST} and \eqref{eqn:error_2_ST}.  Because $\Cr{c1}$ is large and absolute, we may replace $\Cr{c1}$ with the larger constant $\max\{\Cr{c1},(3\Cr{zero-free})^{-1}\log(4 \Cr{int-ineq3}BN)\}$, so
\[
\sum_{x<\mathrm{N}\p\leq x+h}{\bf 1}_{I}(\cos\theta_{\p})\log\mathrm{N}\p\geq b_0\Big(\Big(\frac{3}{4}-o_{h\to\infty}(1)\Big)h-O(BN x^{1-\frac{1}{2A}}(\log x)^3)\Big).
\]
Choosing $h\gg BN x^{1-\frac{1}{2A}}(\log x)^4$, we obtain the lower bound
\begin{equation}
\label{eqn:hoh_linnik}
\sum_{x<\mathrm{N}\p\leq x+h}{\bf 1}_{I}(\cos\theta_{\p})\log\mathrm{N}\p\geq\frac{b_0 h}{2}(1-o_{h\to\infty}(1)).
\end{equation}
when $T$ (hence $x$) is sufficiently large.
\end{proof}

\begin{proof}[Proof of Theorem \ref{thm:least_sato_tate}]
We now address the contribution from Siegel zeros in \eqref{eqn:hoh_linnik}.  Note that if $\pi\in\mathcal{A}_2(K)$ is genuine, then $\mathrm{Sym}^n\pi\in\mathcal{A}_{n+1}(K)$ for $n\leq 4$. Also, by \cite[Theorem B]{Hoffstein}, if $\mathrm{Sym}^j\pi\in\mathcal{A}_{j+1}(K)$ for $j\in \{n-2,n,n+2\}$, then $L(s,\mathrm{Sym}^n\pi,K)$ has no Siegel zero.  We conclude that if $N\geq2$, $I$ can be $\mathrm{Sym}^N$-minorized, and $L(s,\mathrm{Sym}^n\pi,K)$ is automorphic over $K$ for all $n\leq N$, then the only $n\leq N$ for which $L(s,\mathrm{Sym}^n\pi,K)$ might have a Siegel zero are $n\in\{0,N-1,N\}$.

By hypothesis, $I$ can be $\mathrm{Sym}^N$-minorized without admitting Siegel zeros; thus $b_n\leq0$ for each $n\in\{N-1,N\}$ such that $L(s,\mathrm{Sym}^n\pi,K)$ has a Siegel zero.  Thus if $n\in\{N-1,N\}$ and $L(s,\mathrm{Sym}^n\pi,K)$ has a Siegel zero, then such a Siegel zero gives a positive contribution to the lower bound in \eqref{eqn:hoh_linnik}, and we may discard this contribution.  Therefore, if $\zeta_K(s)$ has no Siegel zero, then we can omit the $o(1)$ error term from  \eqref{eqn:hoh_linnik}.  Taking $h=x$, we see that there is an unramified $\p$ such that $\cos\theta_{\p}\in I$ and $\p\leq 2x$, where $x$ is defined at the beginning of the proof of Theorem \ref{thm:hoheisel_sato_tate} after taking $T$ to be a sufficiently large absolute constant.

It remains to bound the maximum of the analytic conductors.  Let $0\leq n\leq N$.  For each unramified $\p$, consider the identity
\[
L_{\p}(s,\pi\otimes\mathrm{Sym}^{n-1}\pi,K)=L_{\p}(s,\mathrm{Sym}^n\pi,K)L_{\p}(s,\mathrm{Sym}^{n-2}\pi,s).
\]
Using \eqref{eqn:zero-count} to relate the arithmetic conductor of each side, we conclude by induction on $n$ that $\log q(\mathrm{Sym}^n\pi)\ll n^3\log q(\pi) \ll N^3 \log q(\pi)$.  (See also Rouse \cite[Lemma 2.1]{Rouse}.  His proof gives an implied constant depending on $[K:\Q]$, but this dependence is easily removed.)  From the shape of $L_{\infty}(s,\mathrm{Sym}^n\pi,K)$ proved by Moreno and Shahidi \cite{MS2}, it follows that
\[
\log\Big(\prod_{j=1}^{(n+1) [K:\Q]}(|\kappa_{\mathrm{Sym}^n \pi}(j)|+3)\Big)\ll n\log\Big(n\prod_{j=1}^{2 [K:\Q]}(|\kappa_{\pi}(j)|+3)\Big),
\]
and the result follows.  In the special case that $\pi$ corresponds to a newform of $\Q$ of squarefree level and trivial nebentypus, Cogdell and Michel \cite{CM} use the local Landglands correspondence to show that $\log q(\mathrm{Sym}^n \pi)=n \log q(\pi)\ll N\log q(\pi)$, which accounts for the claimed improvement.
\end{proof}
\begin{remark}
Suppose now that $\zeta_K(s)$ does have a Siegel zero.  A slight reformulation of the proof of Theorem \ref{thm:hoheisel_sato_tate} with $h=x$ yields the lower bound
\[
\sum_{x<\mathrm{N}\p\leq 2x} {\bf 1}_{I}(\cos\theta_{\p})\log\mathrm{N}\p \geq b_0\Big(\sum_{x<\mathrm{N}\p\leq 2x}\log\mathrm{N}\p-\Cr{int-ineq3}BN e^{-3\Cr{c1}\Cr{zero-free}}x-O(BN  x^{1-\frac{1}{2A}}(\log x)^3)\Big),
\]
Using the lower bound for $\sum_{x<\mathrm{N}\p\leq 2x}\log\mathrm{N}\p$ that follows from \cite[Theorem 5.2]{weiss} (a result which is independent of whether $\zeta_K(s)$ has a Siegel zero), we conclude that there exists a sufficiently large constant $\Cl[abcon]{weiss_const}$ such that
\[
\sum_{x<\mathrm{N}\p\leq 2x} {\bf 1}_{I}(\cos\theta_{\p})\log\mathrm{N}\p \geq b_0\Big(\frac{1}{([K:\Q]^{[K:\Q]}D_K)^{\Cr{weiss_const}}}x-\Cr{weiss_const}BN e^{-3\Cr{c1}\Cr{zero-free}}x-O(BN  x^{1-\frac{1}{2A}}(\log x)^3)\Big).
\]
One can now easily find an effective value of $x$ (which is at least as large as the upper bound in Theorem \ref{thm:hoheisel_sato_tate}) such that there exists an unramified $\p$ such that $\cos\theta_{\p}\in I$ and $\mathrm{N}\p\leq 2x$.  Here, $\Cr{c1}$ needs to be sufficiently large with respect to $B$, $K$, and $N$.
\end{remark}

\section{Proof of Theorem \ref{thm:auto_BV}}
\label{sec:auto_BV}

In this section, all implied constants depend at most on $\q(\pi)$.

\begin{proof}[Proof of Theorem \ref{thm:auto_BV}]

Let $\pi\in\mathcal{A}_2(\Q)$.  Let $Q^5=T\leq x^{\frac{1}{512\Cr{c1'}}},$ and suppose that $x\leq hQ$ and $\log x\leq (\log Q)^2$.  Let $\chi$ be a primitive Dirichlet character modulo $q\leq Q$. It follows from the work of Ramakrishnan and Wang \cite[Theroem A]{RW} that $L(s,(\pi\otimes\chi)\otimes\tilde{\pi},\Q)$ has a Siegel zero $\beta_{\mathrm{exc}}$ if and only if it is inherited from $L(s,\chi,\Q)$.

The proof follows \cite[Section 4]{Gallagher2}.  By arguments similar to those in the proof of Theorem \ref{effective-hoheisel}, it follows that
\[
\sum_{x<n\leq x+h}\Lambda_{\pi\otimes\tilde{\pi}}(n)\chi(n)-\delta(\chi)h+h\xi^{\beta_{\mathrm{exc}}-1}\ll h\Big(\sum_{|\gamma|\leq T}x^{\beta-1}+Q^2/T\Big),
\]
for some $\xi\in[x,x+h]$, where the summation on the right-hand side is over the nontrivial zeros of $L(s,(\pi\otimes\chi)\otimes\tilde{\pi},\Q)$ which are not $\beta_{\mathrm{exc}}$.  Thus
\begin{align}
\label{eqn:triple}
\sum_{q\leq Q}~\sideset{}{^\star}\sum_{\chi\bmod q}\Big|\sum_{x<n\leq x+h}\Lambda_{\pi\otimes\tilde{\pi}}(n)\chi(n)&-\delta(\chi)h+\delta_{q,*}(\chi)h\xi^{\beta_{\mathrm{exc}}-1}\Big|\notag\\
&\ll h\Big(\sum_{q\leq Q}~\sideset{}{^\star}\sum_{\chi\bmod q}~\sum_{|\gamma|\leq T}x^{\beta-1}+Q^4/T\Big).
\end{align}
The triple sum in \eqref{eqn:triple} is bounded by
\[
\log x\int_{\frac{1}{2}}^1 x^{\sigma-1}\sum_{q\leq Q}~\sideset{}{^\star}\sum_{\chi\bmod q}N_{(\pi\otimes\chi)\otimes\tilde{\pi}}(\sigma,T)d\sigma+x^{-1/2}\sum_{q\leq Q}~\sideset{}{^\star}\sum_{\chi\bmod q}N_{(\pi\otimes\chi)\otimes\tilde{\pi}}(1/2,T)
\]
Using Theorem \ref{thm:no_GRC} and recalling our choice of $T$, we bound the above display by
\[
\log x\int_{\frac{1}{2}}^{1-\frac{\Cr{zero-free}}{\mathcal{L}'}}x^{\frac{1}{2}(\sigma-1)}d\sigma+x^{-\frac{1}{4}}\ll x^{-\frac{\Cr{zero-free}}{2\mathcal{L}'}}+x^{-\frac{1}{4}},
\]
where $\mathcal{L}'=256\log(\q(\pi)^2 QT)$.  Since $T=Q^5$, the right-hand side of \eqref{eqn:triple} is bounded as claimed in the statement of Theorem \ref{thm:auto_BV}.
\end{proof}

\appendix
\section{$\mathrm{Sym}^n$-minorants}
We close with two lemmas on $\mathrm{Sym}^n$-minorants.  The first explicitly classifies the intervals which can be $\mathrm{Sym}^4$-minorized, i.e. those intervals we can access unconditionally for any $L(s,\pi,K)$.  The second concerns the asymptotics of the smallest $n$ needed to access the set of primes with $|\lambda_\pi(\p)|>2(1-\delta)$ as $\delta\to 0$ and obtains an improvement over the na\"ive Fourier bound.

\begin{lemma}
\label{lem:sym4-minorized}
Let $\beta_0 = \frac{1+\sqrt{7}}{6} = 0.6076\dots$ and $\beta_{1}=\frac{-1+\sqrt{7}}{6}=0.2742\dots$.  The interval $[a,b]\subseteq [-1,1]$ can be $\mathrm{Sym}^4$-minorized if and only if it satisfies one of the following conditions:
\begin{enumerate}
\item $a=-1$ and $b> -\beta_0$,
\item $-1<a\leq -\beta_0$ and $b>\frac{a+\sqrt{16a^4-11a^2+2}}{2(1-4a^2)}$,
\item $-\beta_0 \leq a \leq -\beta_{1}$ and $b> \frac{-1}{6a}$,
\item $-\beta_{1} \leq a < \beta_{1}$ and $b > \frac{a+\sqrt{16a^4-11a^2+2}}{2(1-4a^2)}$, and
\item $\beta_{1} \leq a < \beta_0$ and $b=1$.
\end{enumerate}
\end{lemma}
\begin{proof}
We begin with sufficiency.  For each case, we list a polynomial $F(x)$ which, for $x\in[-1,1]$, is positive only if $x\in[a,b]$.  We then compute
\[
b_0(F):=\int_{-1}^1 Fd\mu_{ST}
\]
and verify that it is positive.  This is sufficient, since any such $F(x)$ can be scaled to minorize the indicator function.

\begin{enumerate}
\item $F(x)=(x-1)(x-b)(x-\beta_{1})^2$ and $b_0(F)=(b+\beta_0)(\frac{14+\sqrt{7}}{36})$.
\item $F(x)=-(x-a)(x-b)\Big(x+\frac{a+b}{4ab+1}\Big)^2$ and $b_0(F)=\frac{(1-4a^2)b^2-ab+a^2-1/2}{4(4ab+1)}$.
\item $F(x)=(x-1)(x+1)(x-a)(x-b)$ and $b_0(F)=-\frac{3}{4}(ab+\frac{1}{6})$.
\item $F(x)=-(x-a)(x-b)\Big(x+\frac{a+b}{4ab+1}\Big)^2$.
\item $F(x)=(x+1)(x-a)(x+\beta_{1})^2$ and $b_0(F)=(\beta_0-a)(\frac{14+\sqrt{7}}{36})$.
\end{enumerate}

The proof of necessity necessarily involves tedious casework, which we omit.  Let us say only that we consider polynomials $F(x)$, ordered by degree, the number of real roots, and the placement of those roots relative to $a,b,1,$ and $-1$, and in each case we determine conditions under which $b_0(F)>0$.
\end{proof}

\begin{lemma}
\label{lem:extreme-values}
If $n\geq1$, then the set $[-1,-a]\cup[a,1]$ can be $\mathrm{Sym}^{2n}$-minorized if $a < \sqrt{1-\frac{3/2}{n+1}}$.
\end{lemma}
\begin{proof}
We recall the well-known fact that
\[
\int_{-1}^1 x^{2m} \,d\mu_{\mathrm{ST}} = \frac{1}{m+1}{2m\choose m} =: C_m.
\]
Given $n$ and $a$ satisfying the conditions of the lemma, we use the minorant $f_{n,a}(x)=(x^2-a^2)x^{2n-2}$, and we find that
\[
\int_{-1}^1 f_{n,a} \,d\mu_{\mathrm{ST}} = \frac{C_{n-1}}{4^{n-1}} \Big(1-a^2-\frac{3/2}{n+1}\Big).
\]
\end{proof}

\begin{remark}
The Sato-Tate measures of the sets considered in Lemma \ref{lem:extreme-values} satisfy $\mu^{-1} \gg n^{3/2}$, so the minorants in the proof provide a significant improvement over those arising from a na\"ive Fourier approximation.
\end{remark}

\bibliographystyle{abbrv}
\bibliography{GeneralizedLinnik}
\end{document}